\documentclass[12pt,a4paper]{amsart}

\usepackage{amssymb}
\usepackage{latexsym}
\usepackage{amsmath}
\usepackage{euscript}
\usepackage{etoolbox}

\makeatletter
\def\eqnarray{\stepcounter{equation}\let\@currentlabel=\theequation
\global\@eqnswtrue
\tabskip\@centering\let\\=\@eqncr
$$\halign to \displaywidth\bgroup\hfil\global\@eqcnt\z@
  $\displaystyle\tabskip\z@{##}$&\global\@eqcnt\@ne
  \hfil$\displaystyle{{}##{}}$\hfil
  &\global\@eqcnt\tw@ $\displaystyle{##}$\hfil
  \tabskip\@centering&\llap{##}\tabskip\z@\cr}

\def\endeqnarray{\@@eqncr\egroup
      \global\advance\c@equation\m@ne$$\global\@ignoretrue}

\def\@yeqncr{\@ifnextchar [{\@xeqncr}{\@xeqncr[5pt]}}
\makeatother

      \def\dC{{\mathbb C}}

   \def\dN{{\mathbb N}}   
      \def\dR{{\mathbb R}}

      \def\cC{{\mathcal C}}
\def\cD{{\mathcal D}}      
\def\cG{{\mathcal G}}   \def\cH{{\mathcal H}}   
   \def\cK{{\mathcal K}}   \def\cL{{\mathcal L}}
   \def\cN{{\mathcal N}}

\def\cal H{{\mathcal H}}

\def\ran{{\text{\rm ran\,}}}
\def\dom{{\text{\rm dom\,}}}

\def\min{{\text{\rm min}}}

\def\essinf{{\text{\rm essinf}}}

\def\mul{{\text{\rm mul\,}}}
\def\ker{{\text{\rm ker\,}}}

\def\phi{\varphi}

\newtheorem{theorem}{Theorem}[section]
\newtheorem*{thm*}{Theorem}
\newtheorem{proposition}[theorem]{Proposition}
\newtheorem{corollary}[theorem]{Corollary}
\newtheorem{lemma}[theorem]{Lemma}

\theoremstyle{definition}
 
\newtheorem{remark}[theorem]{Remark}
\newtheorem{example}[theorem]{Example}

\numberwithin{equation}{section}

\title[]{Dirichlet-to-Neumann maps on bounded Lipschitz domains}

\author{J. Behrndt}

\address{Institut f\"ur Numerische Mathematik \\
Technische Universit\"at Graz \\
Steyrergasse 30\\
A-8010 Graz\\
Austria}
\email{behrndt@tugraz.at}

\author{A.F.M. ter Elst}
\address{Department of Mathematics  \\
University of Auckland  \\
Private Bag 92019  \\
Auckland 1142  \\
New Zealand}
\email{terelst@math.auckland.ac.nz}

\begin{document}
\bibliographystyle{tom}

\begin{abstract}
The Dirichlet-to-Neumann map associated to an elliptic partial differential equation becomes 
multivalued when the underlying Dirichlet problem is not uniquely solvable.
The main objective of this paper 
is to present a systematic study of the Dirichlet-to-Neumann map and its inverse, the Neumann-to-Dirichlet map, in the framework of linear relations in Hilbert spaces. 
Our treatment is 
inspired by abstract methods from extension theory of symmetric operators, utilizes the general theory of linear relations
and makes use of some deep results on the regularity of the solutions of boundary value problems on bounded Lipschitz domains.
\end{abstract}

\begingroup
\makeatletter
\patchcmd{\@settitle}{\uppercasenonmath\@title}{\large}{}{}
\patchcmd{\@setauthors}{\MakeUppercase}{}{}{}
\makeatother
\maketitle
\endgroup

\section{Introduction}
The Dirichlet-to-Neumann operator is a central object in the analysis of elliptic partial differential equations; it plays a
fundamental role in the classical Calder\'on problem \cite{Cal,Nach1,Nach2,NSU,SU}, is intimately connected with the spectral properties of the associated
partial differential operators, and has a attracted a lot of interest in the recent past, see, e.g. 
\cite{AlpayB,AmrP,AE3,AEKS,ArM2,BehL1,BehL2,BehR,BehR2,BGW,GesM2,GesM3,GesMZ,Mal,Marl} for a small selection of papers of analytic nature.  
In the following let $\Omega$ be a bounded Lipschitz domain in $\dR^n$, where $n\geq 2$, with boundary $\cC$ and consider the differential expression $\cL=-\Delta+V$
on $\Omega$ with $V\in L^\infty(\Omega)$ real valued. 
Under these assumptions it is well known (see for example \cite{McL}, Theorem 4.10)
that for all $\lambda\in\dC$ the Dirichlet problem
\begin{equation}\label{bvp}
\cL f=\lambda f \qquad\text{and}\qquad f\vert_\cC=\varphi
\end{equation}
is solvable for those $\varphi\in H^{1/2}(\cC)$ which satisfy $(\varphi,\partial_\nu h\vert_\cC)_{H^{1/2}(\cC)\times H^{-1/2}(\cC)}=0$
for all solutions $h\in H^1(\Omega)$ of the corresponding homogeneous problem
\begin{equation}\label{bvp2}
\cL h=\lambda h \qquad\text{and}\qquad h\vert_\cC=0.
\end{equation}
Here $\partial_\nu h\vert_\cC\in H^{-1/2}(\cC)$ stands for the normal derivative of $h$ at the boundary $\cC$ of 
$\Omega$ with normal vector pointing outwards and $\cL$ acts as a distribution operator.
In particular, if \eqref{bvp2} has only the trivial solution then for every $\varphi\in H^{1/2}(\cC)$
there exists a unique solution $f_\lambda\in H^1(\Omega)$ of \eqref{bvp}. 
It follows that 
the subspace 
\begin{equation}\label{cauchydata}
 \cD(\lambda):=\bigl\{\{f_\lambda\vert_\cC,\partial_\nu f_\lambda\vert_\cC \}\in H^{1/2}(\cC)\times H^{-1/2}(\cC):
 f_\lambda\in H^1(\Omega) \mbox{ and } \cL f_\lambda=\lambda f_\lambda \bigr\}
\end{equation}
in $H^{1/2}(\cC)\times H^{-1/2}(\cC)$
consisting of the Cauchy data of solutions of \eqref{bvp} can be viewed as the graph of an operator defined on $H^{1/2}(\cC)$
whenever  $\lambda\in\dC$ is such that \eqref{bvp2} has only the trivial solution; this is 
the case if and only if $\lambda$ is not an eigenvalue of the selfadjoint Dirichlet realization 
$$A_D f=\cL f,\qquad \dom A_D=\bigl\{f\in H^1_0(\Omega):-\Delta f+Vf\in L^2(\Omega)\bigr\},$$
in $L^2(\Omega)$.
In other words, for all $\lambda\not\in\sigma_p(A_D)$ the Dirichlet-to-Neumann map
\begin{equation*}
 \cD(\lambda):H^{1/2}(\cC)\rightarrow H^{-1/2}(\cC),\qquad f_\lambda\vert_\cC\mapsto \partial_\nu f_\lambda\vert_\cC,
\end{equation*}
is a well-defined operator on $H^{1/2}(\cC)$. However, the set of Cauchy data is given also for $\lambda\in\sigma_p(A_D)$; 
in this case \eqref{bvp} is solvable only on a subspace of $H^{1/2}(\cC)$ of finite codimension and the solution is not unique. 
It is then natural to view $\cD(\lambda)$ in \eqref{cauchydata} as the graph of a `multivalued operator' defined on a subspace of $H^{1/2}(\cC)$
`mapping' into  $H^{-1/2}(\cC)$. This point of view was also taken in the recent publications \cite{AEKS} and \cite{ArM2}, where the restriction
of $\cD(\lambda)$ to $L^2(\cC)$ was shown to be a selfadjoint linear relation in $L^2(\cC)$ which is semibounded from below. 
In \cite{ArM2} an argument relying on a Galerkin approximation method
was employed, and in \cite{AEKS} general form methods based on a Fredholm alternative and compact embeddings were used. Both approaches are of more 
extrinsic nature and do not allow a detailed study of spectral and mapping properties of the selfadjoint linear relation and its operator part.

The main objective of the present paper is to present a systematic and more intrinsic study of Dirichlet-to-Neumann maps  in the framework
of linear relations in Hilbert spaces. The calculus of linear relations is a very useful and convenient tool even when studying operators,
e.g. the inverse of a selfadjoint operator is always a selfadjoint linear relation, and hence admits a spectral function and a
functional calculus similar to the ones of selfadjoint operators. We briefly review some elements in the theory of symmetric and selfadjoint 
linear relations in the appendix.  In Sections~\ref{sec2} and \ref{sec3} we first recall some basic facts on Sobolev 
spaces on Lipschitz domains, trace operators, and the selfadjoint Dirichlet operator $A_D$ and Neumann operator $A_N$ associated with the differential expression 
$\cL=-\Delta+V$ in $L^2(\Omega)$. 
Section~\ref{sec4} is devoted to the Dirichlet-to-Neumann map $\cD(\lambda)$ and its inverse, the Neumann-to-Dirichlet map $\cN(\lambda)$,
viewed as linear relations in $H^{1/2}(\cC)\times H^{-1/2}(\cC)$ and $H^{-1/2}(\cC)\times H^{1/2}(\cC)$, respectively. 
After discussing some elementary properties of their domains, multivalued parts, kernels, and ranges, we establish a connection between
$\cD(\lambda)$ and $\cD(\mu)$ (and, similarly for $\cN(\lambda)$ and $\cN(\mu)$) for all $\lambda,\mu\in\dC$ in Theorem~\ref{dnthm} and 
Corollary~\ref{dncor}, and we prove a variant of a Krein type formula for the resolvent difference of $A_D$ and $A_N$ in Theorem~\ref{resdiffthm}.
Such formulae are known under the additional assumption $\lambda,\mu\not\in(\sigma_p(A_D)\cup\sigma_p(A_N))$, and it is remarkable that
they remain true for all $\lambda,\mu\in\dC$ when interpreted in the sense of linear relations.
The origin of these correspondences is in abstract extension theory of symmetric operators in Hilbert spaces, where, roughly speaking, the functions 
$\lambda\mapsto \cD(\lambda)$ and $\lambda\mapsto\cN(\lambda)$ can be viewed as so-called $Q$-functions or Weyl functions; 
cf. \cite{BehL1,BehL2,DHMS,DerkachM,LanT}. We wish to emphasize that the considerations and results in Section~\ref{sec4} 
are mainly based on Green's identity and elementary computations of mostly algebraic nature, and that no deeper results on elliptic regularity
or compactness properties of the involved trace mappings and embeddings are employed. This changes dramatically in Section~\ref{sec5}, where the restrictions 
$$D(\lambda)=\cD(\lambda)\cap \bigl(L^2(\cC)\times L^2(\cC)\bigr)\qquad\text{and}\qquad N(\lambda)=\cN(\lambda)\cap \bigl(L^2(\cC)\times L^2(\cC)\bigr)$$
are considered as linear relations in $L^2(\cC) \times L^2(\cC)$. 
An essential ingredient in our further analysis are results due to Jerison and Kenig \cite{JK2,JK},
and Gesztesy and Mitrea \cite{GesM2,GesM3} on the $H^{3/2}(\Omega)$-regularity of the functions in $\dom A_D$ and $\dom A_N$, and the solvability of \eqref{bvp} in 
$H^{3/2}(\Omega)$ for boundary data $\varphi\in H^1(\cC)$. 
We specify the domains of $D(\lambda)$ and $N(\lambda)$ in Theorem~\ref{thmdn},
and observe the interesting fact that their kernels and multivalued parts coincide with those of $\cD(\lambda)$ and $\cN(\lambda)$. 
The main results in Section~\ref{sec5} are Theorems~\ref{dnmapsa} and \ref{dnmapsa2}, 
where it is shown that if $\lambda \in \dR$ then $D(\lambda)$ and $N(\lambda)$
are selfadjoint relations in $L^2(\cC)$ with finitely many negative eigenvalues, the operator part of $N(\lambda)$ is a compact selfadjoint operator,
and the operator part of $D(\lambda)$ is an unbounded selfadjoint operator with discrete spectrum. These theorems can be viewed as extensions and refinements of 
some results in \cite{AEKS,ArM2}.
However, the strategy for the proofs in Section~\ref{sec5} is very much different from the methods used in \cite{AEKS,ArM2}. Here we rely on an explicit connection
of the Neumann-to-Dirichlet map $N(\lambda)$ with $(A_N-\lambda)^{-1}$, where the latter is a selfadjoint relation with multivalued part $\ker(A_N-\lambda)$
and compact operator part with finitely many negative eigenvalues. We then 
deduce spectral and mapping properties of $N(\lambda)$ from those of $(A_N-\lambda)^{-1}$ via perturbation arguments and an abstract result on the
selfadjointness of a product of bounded operators and a selfadjoint relation in Proposition~\ref{ttprop}. After establishing the selfadjointness
and the spectral properties of the Neumann-to-Dirichlet map the corresponding facts for the Dirichlet-to-Neumann map 
follow immediately from $D(\lambda)=N(\lambda)^{-1}$. We finally mention that such type of results were successfully applied in the special 
case of the Laplacian in \cite{ArM2} to prove strict inequalities between Dirichlet and Neumann eigenvalues.

\vspace{5mm}

\noindent
{\bf Acknowledgements}. J. Behrndt 
is most grateful for the stimulating research stay and the hospitality at the University of Auckland in January 2014 
where parts of this paper were written.
This work is supported by the Austrian 
Science Fund (FWF), project
P 25162-N26 and 
part of this work is supported by the Marsden Fund Council from Government funding, 
administered by the Royal Society of New Zealand.

\section{Lipschitz domains, Sobolev spaces and trace operators}\label{sec2}

Let $\Omega\subset\dR^n$, where $n\geq 2$, be a bounded Lipschitz domain with boundary $\cC$. 
By $H^s(\Omega)$ and $H^s(\cC)$ we denote the Sobolev spaces of order $s \geq 0$ on $\Omega$ and 
$\cC$, respectively, and by $H_0^s(\Omega)$ the closure of the set of $C^\infty$-functions with compact support in $\Omega$ 
with respect to the $H^s$-norm.
Further, $H^{-s} (\cC)$ denotes the dual space of $H^{s} (\cC)$; the corresponding extension of the $L^2(\cC)$ inner product
onto $H^{s} (\cC)\times H^{-s} (\cC)$ is denoted by $(\cdot,\cdot)_{H^{s} (\cC)\times H^{-s} (\cC)}$.
We write $u |_\cC\in H^{1/2}(\cC)$ for the trace of $u \in H^1(\Omega)$ at 
the boundary $\cC$ and if $\Delta u \in L^2(\Omega)$ then we set
$\partial_\nu u |_\cC\in H^{-{1/2}}(\cC)$ for the Neumann trace of $u$ 
at $\cC$, see, e.g. \cite{McL}, Theorem~3.37 and Lemma~4.3.
Recall that $\partial_\nu u\vert_\cC$ is the unique function in $H^{-{1/2}}(\cC)$ which satisfies
\begin{equation}\label{ederi}
(\partial_\nu u|_\cC,v|_\cC)_{H^{-{1/2}}(\cC) \times H^{{1/2}}(\cC)}
=(\Delta u,v)_{H^{-1}(\Omega) \times H^1(\Omega)} +(\nabla u,\nabla v)_{L^2(\Omega)^n}
\end{equation}
for all $v\in H^1(\Omega)$,
see \cite{McL}, Lemma~4.3.
Note also that $H^1_0(\Omega)$ coincides with the kernel of the trace operator $u\mapsto u\vert_\cC$ on $H^1(\Omega)$.

\section{Schr\"{o}dinger operators with Dirichlet and Neumann boundary conditions}\label{sec3}

Let $V\in L^\infty(\Omega)$ be a real valued function and consider the differential expression
\[
 \cL:=-\Delta + V.
\]
The first Green's identity states that 
\begin{equation*}
(\cL u,v)_{L^2(\Omega)}
=(\nabla u,\nabla v)_{L^2(\Omega)^n}+(Vu,v)_{L^2(\Omega)}
   -(\partial_\nu u\vert_\cC,v\vert_\cC)_{H^{-1/2}(\cC) \times H^{1/2}(\cC)}
\end{equation*}
for all $u,v\in H^1(\Omega)$ such that $\cL u,\cL v\in L^2(\Omega)$.
The selfadjoint operators $A_D$ and $A_N$ in $L^2(\Omega)$ with Dirichlet and Neumann boundary conditions
are defined as the representing operators of the closed symmetric lowerbounded sesquilinear forms
\[
\begin{split}
\mathfrak a_D[u,v]&=(\nabla u,\nabla v)_{L^2(\Omega)^n}+(Vu,v)_{L^2(\Omega)},\qquad u,v\in H^1_0(\Omega),\\
\mathfrak a_N[u,v]&=(\nabla u,\nabla v)_{L^2(\Omega)^n}+(Vu,v)_{L^2(\Omega)},\qquad u,v\in H^1(\Omega).
\end{split}
\]
It follows that the selfadjoint operators $A_D$ and $A_N$ are given by
\begin{equation*}
\begin{split}
A_D=-\Delta+V,&\qquad\dom A_D=\bigl\{u\in H^1_0(\Omega):\cL u\in L^2(\Omega)\bigr\},\\
A_N=-\Delta+V,&\qquad\dom A_N=\bigl\{u\in H^1(\Omega):\partial_\nu u\vert_\cC=0 \mbox{ and } \cL u\in L^2(\Omega)\bigr\}.
\end{split}
\end{equation*}
Moreover, it follows from the lower boundedness of the forms $\mathfrak a_D$ and $\mathfrak a_N$ that 
the operators $A_D$ and $A_N$ are lower bounded with lower bound $\essinf\,V$.

\section{Dirichlet-to-Neumann and Neumann-to-Dirichlet maps}\label{sec4}

The Dirichlet-to-Neumann map $\cD(\lambda)$ and the Neumann-to-Dirichlet map $\cN(\lambda)$ associated to the differential expression 
$\cL-\lambda$ are defined for
all $\lambda\in\dC$ as subspaces of $H^{1/2}(\cC)\times H^{-1/2}(\cC)$ and
$H^{-1/2}(\cC)\times H^{1/2}(\cC)$, respectively, by 
\[
\begin{split}
\cD(\lambda):=\bigl\{\{f_\lambda\vert_\cC,\partial_\nu f_\lambda\vert_\cC\} \in H^{1/2}(\cC)\times H^{-1/2}(\cC):
 f_\lambda\in H^1(\Omega) \mbox{ and } \cL f_\lambda=\lambda f_\lambda \bigr\},\\
\cN(\lambda):=\bigl\{\{\partial_\nu f_\lambda\vert_\cC,f_\lambda\vert_\cC\} \in H^{-1/2}(\cC)\times H^{1/2}(\cC):
 f_\lambda\in H^1(\Omega) \mbox{ and } \cL f_\lambda=\lambda f_\lambda \bigr\}.
\end{split}
\]
See Appendix~\ref{Sappendix} for a short introduction into the theory of linear relations.
Clearly, 
$$\cD(\lambda)^{-1}=\cN(\lambda)\qquad\text{and}\qquad\cD(\lambda)=\cN(\lambda)^{-1}$$ 
in the sense of linear relations, and
\begin{equation}\label{emulker}
\begin{split}
 \ker \cD(\lambda)&=\mul\cN(\lambda)=\bigl\{f_\lambda\vert_\cC:f_\lambda\in \ker (A_N-\lambda) \bigr\}\subset H^{1/2}(\cC),\\
 \ker \cN(\lambda)&=\mul\cD(\lambda)=\bigl\{\partial_\nu f_\lambda\vert_\cC:f_\lambda\in \ker (A_D-\lambda) \bigr\}\subset H^{-1/2}(\cC).
\end{split}
\end{equation}
It will be shown later in Theorem~\ref{thmdn} that in fact $\ker \cD(\lambda)\subset H^1(\cC)$ and $\ker \cN(\lambda)\subset L^2(\cC)$.
We next characterise the domains of the Dirichlet-to-Neumann map $\cD(\lambda)$ 
and the Neumann-to-Dirichlet map $\cN(\lambda)$.

\begin{proposition}\label{domdomprop}
For all $\lambda\in\dC$ the domains of $\cD(\lambda)$ and $\cN(\lambda)$ are
 \begin{equation}\label{edomgammad}
 \begin{split}
 \dom \cD(\lambda)&= \bigl\{\varphi\in H^{1/2}(\cC) : 
   (\varphi,\partial_\nu f_\lambda\vert_\cC)_{H^{1/2}(\cC) \times H^{-1/2}(\cC)}=0 \mbox{ for all }
 f_\lambda\in\ker(A_D-\lambda) \bigr\},\\
 \dom \cN(\lambda)&= \bigl\{\psi\in H^{-1/2}(\cC) : 
   (\psi,f_\lambda\vert_\cC)_{H^{-1/2}(\cC) \times H^{1/2}(\cC)}=0 \mbox{ for all }
 f_\lambda\in\ker(A_N-\lambda) \bigr\},
 \end{split}
\end{equation}
and, in particular, 
\begin{equation}\label{showthis}
 \dom \cD(\lambda)=H^{1/2}(\cC)  \mbox{ if and only if } \lambda\in\rho(A_D)
\end{equation}
and
\begin{equation}\label{dontshowthis}
 \dom \cN(\lambda)=H^{-1/2}(\cC)  \mbox{ if and only if } \lambda\in\rho(A_N).
\end{equation}
\end{proposition}

\begin{proof}
The equalities (\ref{edomgammad}) follow from \cite{McL}, Theorem 4.10. 
For \eqref{showthis} assume first that $\dom \cD(\lambda)=H^{1/2}(\cC)$. Then it follows
that $\partial_\nu f_\lambda\vert_\cC=0$ for all $f_\lambda\in\ker(A_D-\lambda)$. 
Moreover, as $f_\lambda\vert_\cC=0$ for all $f_\lambda\in\ker(A_D-\lambda)$
we conclude that $f_\lambda=0$ from the unique continuation property. This implies $\ker(A_D-\lambda)=\{0\}$ and hence $\lambda\in\rho(A_D)$. The converse implication in
\eqref{showthis} is immediate. The equivalence \eqref{dontshowthis} follows from a very similar reasoning.
\end{proof}

The next lemma shows for which $\lambda\in\dC$ the linear relations $\cD(\lambda)$ and $\cN(\lambda)$ are (the graphs of) operators mapping from
$H^{1/2}(\cC)$ to $H^{-1/2}(\cC)$ and
$H^{-1/2}(\cC)$ to $H^{1/2}(\cC)$, respectively. 

\begin{lemma}\label{mullem}
Let  $\lambda\in\dC$. 
Then
\begin{itemize}
 \item [(i)] $\mul\cD(\lambda)=\{0\}$ if and only if $\lambda\not\in\sigma_p(A_D)$,
 \item [(ii)] $\mul\cN(\lambda)=\{0\}$ if and only if $\lambda\not\in\sigma_p(A_N)$,
\end{itemize} 
and,
\begin{itemize}
 \item [(iii)] $\ker\cN(\lambda)\not=\{0\}$ if and only if $\lambda\in\sigma_p(A_D)$,
 \item [(iv)] $\ker\cD(\lambda)\not=\{0\}$ if and only if $\lambda\in\sigma_p(A_N)$.
\end{itemize} 
\end{lemma}

\begin{proof}
We verify item (i) only; the proof of item (ii) is analogous, items (iii) and (iv) follow from \eqref{emulker} and (i)-(ii).
Assume 
that $\mul\cD(\lambda)=\{0\}$ and let $f_\lambda\in \ker(A_D-\lambda)$, 
that is, $f_\lambda\in H^1(\Omega)$ satisfies $\cL f_\lambda=\lambda f_\lambda$ and $f_\lambda\vert_\cC=0$. 
As $\{f_\lambda\vert_\cC,\partial_\nu f_\lambda\vert_\cC\}=\{0,\partial_\nu f_\lambda\vert_\cC\}\in \cD(\lambda)$
we conclude that $\partial_\nu f_\lambda\vert_\cC=0$ from the assumption $\mul\cD(\lambda)=\{0\}$.
The unique continuation property implies $f_\lambda=0$
and hence $\ker(A_D-\lambda)=\{0\}$, that is, $\lambda\not\in\sigma_p(A_D)$.
Conversely, if $\lambda\not\in\sigma_p(A_D)$ then
it follows from \eqref{emulker} that $\mul\cD(\lambda)=\{0\}$.
\end{proof}

If $u, v \in H^1(\Omega)$ satisfy $\cL u, \cL v \in L^2(\Omega)$ then 
the second Green identity states
\begin{equation*} 
 (\cL u, v)_{L^2(\Omega)} - (u, \cL v)_{L^2(\Omega)}=\bigl( u |_\cC, \partial_\nu v |_\cC \bigr)_{H^{1/2}(\cC) \times H^{-1/2}(\cC)}
 - \bigl( \partial_\nu u |_\cC, v |_\cC \bigr)_{H^{-1/2}(\cC) \times H^{1/2}(\cC)},  
\end{equation*}
see, e.g.,~\cite{McL}, Theorem~4.4(iii).
As a consequence one deduces the next lemma.

\begin{lemma}\label{greenlem}
Let $\lambda,\mu\in\dC$ and suppose that 
$\{f_\lambda\vert_\cC,\partial_\nu f_\lambda\vert_\cC\}\in\cD(\lambda)$ and 
$\{g_\mu\vert_\cC,\partial_\nu g_\mu\vert_\cC\}\in\cD(\mu)$, or, equivalently,
$\{\partial_\nu f_\lambda\vert_\cC,f_\lambda\vert_\cC\}\in\cN(\lambda)$ and
$\{\partial_\nu g_\mu\vert_\cC,g_\mu\vert_\cC\}\in\cN(\mu)$.
Then
\begin{equation*}
\bigl(\partial_\nu f_\lambda\vert_\cC,g_\mu\vert_\cC\bigr)_{H^{-1/2}(\cC) \times H^{1/2}(\cC)} 
-\bigl(f_\lambda\vert_\cC,\partial_\nu g_\mu\vert_\cC\bigr)_{H^{1/2}(\cC) \times H^{-1/2}(\cC)}=(\overline\mu-\lambda)(f_\lambda,g_\mu)_{L^2(\Omega)}.
\end{equation*}
\end{lemma}

We note that \eqref{showthis} and Lemma~\ref{greenlem} imply 
\[
 (\cD(\lambda)\varphi,\psi)_{H^{-1/2}(\cC)\times H^{1/2}(\cC)}=(\varphi,\cD(\overline\lambda)\psi)_{H^{1/2}(\cC)\times H^{-1/2}(\cC)}
\]
for all $\lambda\in\rho(A_D)$ and all $\varphi,\psi\in H^{1/2}(\cC)$. 
Therefore for all $\lambda\in\rho(A_D)$ the operator
$$\cD(\lambda) \colon H^{1/2}(\cC)\rightarrow H^{-1/2}(\cC)$$ 
is closed and hence bounded by the closed
graph theorem. Similarly it follows from \eqref{dontshowthis} and Lemma~\ref{greenlem} that 
$\cN(\lambda) \colon H^{-1/2}(\cC)\rightarrow H^{1/2}(\cC)$ is a bounded operator for all $\lambda\in\rho(A_N)$.

For all $\lambda\in\dC$ define the subspaces $\gamma_\cD(\lambda)$ 
of $H^{1/2}(\cC)\times L^2(\Omega)$ and 
$\gamma_\cN(\lambda)$ of $H^{-1/2}(\cC)\times L^2(\Omega)$ by
\[
\begin{split}
\gamma_\cD(\lambda):=&\bigl\{\{f_\lambda\vert_\cC,f_\lambda\} \in H^{1/2}(\cC)\times L^2(\Omega):
    f_\lambda\in H^1(\Omega) \mbox{ and } \cL f_\lambda=\lambda f_\lambda \bigr\},\\
\gamma_\cN(\lambda):=&\bigl\{\{\partial_\nu f_\lambda\vert_\cC,f_\lambda \} \in H^{-1/2}(\cC)\times L^2(\Omega):
   f_\lambda\in H^1(\Omega) \mbox{ and } \cL f_\lambda=\lambda f_\lambda \bigr\}.
\end{split}
\]
Note that $\ran\gamma_\cD(\lambda)$ and $\ran\gamma_\cN(\lambda)$ are contained in $H^1(\Omega)$.
Obviously, we have
\begin{equation*}
  \dom \gamma_\cD(\lambda) =\dom \cD(\lambda),\qquad
  \mul \gamma_\cD(\lambda) =\ker(A_D-\lambda),
\end{equation*}
and 
\begin{equation}\label{edomdom2}
  \dom \gamma_\cN(\lambda) =\dom \cN(\lambda),\qquad  \mul \gamma_\cN(\lambda) =\ker(A_N-\lambda).
\end{equation}
Furthermore, it is clear that $\ker\gamma_\cD(\lambda)=\{0\}$ and $\ker\gamma_\cN(\lambda)=\{0\}$ for all $\lambda\in\dC$.

In the next lemma it is shown how $\gamma_\cD(\lambda)$ and $\gamma_\cD(\mu)$ 
are related to each other for different $\lambda,\mu\in\dC$.
If both points $\lambda$ and $\mu$ are not in the spectrum of $A_D$ 
these facts are known, see, e.g., \cite{BehR}, Lemma 2.4.
Similar results are valid for $\gamma_\cN(\lambda)$, $\gamma_\cN(\mu)$ and $A_N$.

\begin{lemma}\label{gammalambdamu}
Let $\lambda,\mu\in\dC$.
Then
\begin{equation*}
\gamma_\cD(\lambda)\cap\bigl( \dom\gamma_\cD(\mu)\times L^2(\Omega)\bigr)=\bigl(I+(\lambda-\mu)(A_D-\lambda)^{-1}\bigr)\gamma_\cD(\mu)
\end{equation*}
and
\begin{equation}\label{euio}
 \gamma_\cN(\lambda)\cap\bigl(\dom\gamma_\cN(\mu)\times L^2(\Omega)\bigr)=\bigl(I+(\lambda-\mu)(A_N-\lambda)^{-1}\bigr)\gamma_\cN(\mu).
\end{equation}
\end{lemma}

\begin{proof}
We verify the statement for $\gamma_\cN$; the proof for $\gamma_\cD$ is completely analogous.
We may assume that $\lambda\not=\mu$; 
otherwise the statement is obviously true.
Note first that in the sense of linear relations 
we have
\begin{eqnarray*}
\lefteqn{
\bigl(I+(\lambda-\mu)(A_N-\lambda)^{-1}\bigr)\gamma_\cN(\mu)
} \hspace{15mm} \\*
& = & 
\left\{\bigl\{ \partial_\nu f_\mu\vert_\cC , f_\mu +(\lambda-\mu) h\bigr \}:
\begin{array}{l} \mbox{there exist } f_\mu\in H^1(\Omega) \mbox{ and } h \in \dom A_N \\
\mbox{such that } \cL f_\mu=\mu f_\mu \mbox{ and } f_\mu=(A_N-\lambda)h \end{array}\right\} .
\end{eqnarray*}
For the inclusion $\supset$ in \eqref{euio} 
let $f_\lambda:=f_\mu +(\lambda-\mu) h$ with $\cL f_\mu=\mu f_\mu$, $f_\mu\in H^1(\Omega)$, 
and $f_\mu=(A_N-\lambda)h$ for some
$h\in\dom A_N$.
Then we have $f_\lambda\in H^1(\Omega)$,
$$(\cL-\lambda)f_\lambda=(\cL-\lambda)\bigl(f_\mu +(\lambda-\mu) h\bigr)=(\mu-\lambda)f_\mu+(\lambda-\mu)(A_N-\lambda)h=0$$
and $\partial_\nu f_\lambda\vert_\cC=\partial_\nu(f_\mu +(\lambda-\mu) h)\vert_\cC=\partial_\nu f_\mu\vert_\cC$.
Therefore
\begin{equation*}
 \bigl\{ \partial_\nu f_\mu\vert_\cC , f_\mu +(\lambda-\mu) h\bigr\}=
\bigl\{\partial_\nu f_\lambda\vert_\cC , f_\lambda\bigr\}\in\gamma_\cN(\lambda).
\end{equation*}
For the inclusion $\subset$ in \eqref{euio} consider $\{\partial_\nu f_\lambda\vert_\cC,f_\lambda\}\in\gamma_\cN(\lambda)$ 
and suppose $\partial_\nu f_\lambda\vert_\cC
\in\dom\gamma_\cN(\mu)$. Then 
$f_\lambda \in H^1(\Omega)$, $\cL f_\lambda=\lambda f_\lambda$, and there exists an $f_\mu\in H^1(\Omega)$ such that $\cL f_\mu=\mu f_\mu$ and $\partial_\nu f_\mu\vert_\cC=\partial_\nu f_\lambda\vert_\cC$.
It follows that 
$$h:=\frac{f_\lambda-f_\mu}{\lambda-\mu}\in\dom A_N$$ 
and $f_\mu+(\lambda-\mu) h=f_\lambda$.
Therefore
$$\bigl\{ \partial_\nu f_\lambda\vert_\cC , f_\lambda\bigr\}=
\bigl\{ \partial_\nu f_\mu\vert_\cC, f_\mu+(\lambda-\mu) h \bigr\}
\in\bigl(I+(\lambda-\mu)(A_N-\lambda)^{-1}\bigr)\gamma_\cN(\mu)$$
as required.
\end{proof}

In the next lemma the adjoints of $\gamma_\cD(\lambda)$ and $\gamma_\cN(\lambda)$ are computed.
Recall that  $\gamma_\cD(\lambda)$ is  a linear relation in $H^{1/2}(\cC)\times L^2(\Omega)$.
So the adjoint $\gamma_\cD(\lambda)^\prime$ 
is a linear relation in $L^2(\Omega)\times H^{-1/2}(\cC)$.
Similarly $\gamma_\cN(\lambda)$ is a linear relation in $H^{-1/2}(\cC)\times L^2(\Omega)$
and its adjoint $\gamma_\cN(\lambda)^\prime$ is a linear relation in $L^2(\Omega)\times H^{1/2}(\cC)$.

\begin{lemma}\label{gammaadj}
Let $\lambda\in\dC$.
Then
\begin{equation*}
  \gamma_\cD(\lambda)^\prime =\bigl\{\{(A_D-\overline\lambda)g,-\partial_\nu g\vert_\cC\}:g\in\dom A_D\bigr\} 
\end{equation*}  
  and 
  \begin{equation}\label{eaaa}
  \gamma_\cN(\lambda)^\prime =\bigl\{\{(A_N-\overline\lambda)g, g\vert_\cC\}:g\in\dom A_N\bigr\}.
\end{equation}
\end{lemma}

\begin{proof}
We prove \eqref{eaaa}. 
First the inclusion $\supset$ will be shown.
Let $g\in\dom A_N$.
We shall show that 
$\{(A_N-\overline\lambda)g, g\vert_\cC\} \in\gamma_\cN(\lambda)^\prime$.
Indeed, one has $\partial_\nu g\vert_\cC=0$
and for any
 $\{\partial_\nu f_\lambda\vert_\cC,f_\lambda\}\in\gamma_\cN(\lambda)$
we compute with the help of Green's identity that 
\begin{equation}\label{ecompu}
\begin{split}
(f_\lambda,(A_N-\overline\lambda)g)_{L^2(\Omega)}
&=(f_\lambda,A_N g)_{L^2(\Omega)} 
   -(\cL f_\lambda,g)_{L^2(\Omega)}\\
&=(\partial_\nu f_\lambda\vert_\cC,g\vert_\cC)_{H^{-1/2}(\cC) \times H^{1/2}(\cC)}
   -(f_\lambda\vert_\cC,\partial_\nu g\vert_\cC)_{H^{1/2}(\cC) \times H^{-1/2}(\cC)}\\
&=(\partial_\nu f_\lambda\vert_\cC,g\vert_\cC)_{H^{-1/2}(\cC) \times H^{1/2}(\cC)}.
\end{split}
\end{equation}
This implies that
$\{(A_N-\overline\lambda)g, g\vert_\cC\}\in\gamma_\cN(\lambda)^\prime$. 

For the inclusion $\subset$ we have to check that for any element 
$\{h,\varphi\}\in\gamma_\cN(\lambda)^\prime$
there exists a $g\in\dom A_N$
such that
\begin{equation}\label{etocheck}
 \{h,\varphi\}=\bigl\{(A_N-\overline\lambda)g,g\vert_\cC\bigr\}.
\end{equation}
Since
\[
\dom\gamma_\cN(\lambda)^\prime \subset \bigl( \mul\gamma_\cN(\lambda)\bigr)^\bot=\bigl(\ker(A_N-\lambda)\bigr)^\bot=\ran(A_N-\overline\lambda)
\]
there exists a $k\in\dom A_N$ with 
\begin{equation}\label{ehk}
h=(A_N-\overline\lambda)k.
\end{equation} 
Hence 
$\{(A_N-\overline\lambda)k,\varphi\} = \{h,\varphi\}\in\gamma_\cN(\lambda)^\prime$
and for all $\{\partial_\nu f_\lambda\vert_\cC,f_\lambda\}
\in\gamma_\cN(\lambda)$ we have by the definition of the adjoint 
\begin{equation}\label{e123}
(f_\lambda,(A_N-\overline\lambda)k)_{L^2(\Omega)}
=(\partial_\nu f_\lambda\vert_\cC,\varphi)_{H^{-1/2}(\cC) \times H^{1/2}(\cC)}.
\end{equation}
On the other hand the same calculation as in \eqref{ecompu} with Green's identity yields
\begin{equation}\label{e456}
(f_\lambda,(A_N-\overline\lambda)k)_{L^2(\Omega)}
=(\partial_\nu f_\lambda\vert_\cC,k\vert_\cC)_{H^{-1/2}(\cC) \times H^{1/2}(\cC)}.
\end{equation}
From \eqref{e123} and \eqref{e456} we obtain
$(\psi,\varphi-k\vert_\cC)_{H^{-1/2}(\cC) \times H^{1/2}(\cC)}=0$ 
for all $\psi\in\dom\gamma_\cN(\lambda)$
and hence \eqref{edomdom2} and Proposition~\ref{domdomprop} imply that there exists a 
$k_\lambda\in\ker(A_N-\lambda)$ such that
\begin{equation}\label{ehkspur}
 k_\lambda\vert_\cC=\varphi-k\vert_\cC.
\end{equation}
Note that $k_\lambda=0$ if $\lambda\not\in\sigma_p(A_N)$, in particular, $k_\lambda=0$ if $\lambda\in\dC\setminus\dR$.
It follows from \eqref{ehk} and \eqref{ehkspur} that $g:=k+k_\lambda\in\dom A_N$ satisfies 
$h=(A_N-\overline\lambda)g$ and $g\vert_\cC=\varphi$.
We have shown \eqref{etocheck} and hence \eqref{eaaa} is proved. 
\end{proof}

As an immediate consequence of Lemma~\ref{gammaadj} we have
\[
 \begin{split}
  \mul\gamma_\cD(\lambda)^\prime =\bigl\{\partial_\nu g\vert_\cC: g\in\ker(A_D-\overline\lambda)\bigr\} 
 \end{split}
\]
and 
\[
 \mul\gamma_\cN(\lambda)^\prime =\bigl\{g\vert_\cC: g\in\ker(A_N-\overline\lambda)\bigr\}.
\]

Note also that if $\lambda\in\rho(A_D)$ then
\[
 \gamma_\cD(\lambda)^\prime \colon L^2(\Omega)\rightarrow H^{-1/2}(\cC),\quad h\mapsto -\partial_\nu \bigl((A_D-\overline\lambda)^{-1}h\bigr)\vert_\cC,
\]
is a closed operator defined on the whole space $L^2(\Omega)$, and hence 
$\gamma_\cD(\lambda)^\prime$ is a bounded operator; cf.\ \cite{BehR}, Lemma 2.4.
Similarly,
\[
\gamma_\cN(\lambda)^\prime \colon L^2(\Omega)\rightarrow H^{1/2}(\cC),\quad h\mapsto \bigl((A_N-\overline\lambda)^{-1}h\bigr)\vert_\cC,
\]
is a bounded operator  for all $\lambda\in\rho(A_N)$.

\begin{theorem}\label{dnthm}
Let $\lambda,\mu\in\dC$. Then
\begin{equation}\label{ednrel}
 \cD(\lambda)-\cD(\overline\mu)=(\overline\mu-\lambda)\gamma_\cD(\mu)^\prime\gamma_\cD(\lambda)
\end{equation}
and
\begin{equation}\label{enmrel}
 \cN(\lambda)-\cN(\overline\mu)=(\lambda-\overline\mu)\gamma_\cN(\mu)^\prime\gamma_\cN(\lambda).
\end{equation}
\end{theorem}

\begin{proof}
Only the assertion \eqref{enmrel} will be verified.
The proof of \eqref{ednrel} is similar.
We may assume that $\lambda \neq \overline \mu$.
We show the inclusion $\subset$ in \eqref{enmrel} first.
Let $\{\varphi,\psi\}\in\cN(\lambda)-\cN(\overline\mu)$, that is,
there are $f_\lambda, g_{\overline\mu}\in H^1(\Omega)$ such that 
$\cL f_\lambda=\lambda f_\lambda$, $\cL g_{\overline\mu}=\overline\mu g_{\overline\mu}$,
\[
\varphi=\partial_\nu f_\lambda\vert_\cC
=\partial_\nu g_{\overline\mu}\vert_\cC\quad\text{and}\quad\psi
=f_\lambda\vert_\cC- g_{\overline\mu}\vert_\cC.
\]
In particular,
$\{\varphi,f_\lambda\}\in\gamma_\cN(\lambda)$.
Observe
that
\[
 h:=\frac{f_\lambda -g_{\overline\mu}}{\lambda-\overline\mu}\in\dom A_N
\quad\text{and}\quad (A_N-\overline\mu)h=f_\lambda.
\]
One concludes from Lemma~\ref{gammaadj} that  $\{f_\lambda, h\vert_\cC\}=\{(A_N-\overline\mu)h, h\vert_\cC\}\in\gamma_\cN(\mu)^\prime$
and therefore we have $\{\varphi,h\vert_\cC\}\in\gamma_\cN(\mu)^\prime\gamma_\cN(\lambda)$.
As $(\lambda-\overline\mu)h\vert_\cC=f_\lambda\vert_\cC- g_{\overline\mu}\vert_\cC=\psi$ we obtain
\[
\{ \varphi, \psi \}=
\bigl\{ \varphi,  (\lambda-\overline\mu)h\vert_\cC\bigr\}\in(\lambda-\overline\mu)\gamma_\cN(\mu)^\prime\gamma_\cN(\lambda).
\]
This proves the inclusion $\subset$ in \eqref{enmrel}.

Let us now verify the inclusion $\supset$ in \eqref{enmrel}.
Assume that $\{\varphi,\psi\}\in(\lambda-\overline\mu)\gamma_\cN(\mu)^\prime\gamma_\cN(\lambda)$.
Then there exists an $f_\lambda\in H^1(\Omega)$ such that 
$\cL f_\lambda=\lambda f_\lambda$,
$\varphi=\partial_\nu f_\lambda\vert_\cC$, 
$\{\varphi,f_\lambda\}\in\gamma_\cN(\lambda)$
and 
$$\bigl\{f_\lambda , (\lambda-\overline\mu)^{-1} \psi \bigr\}\in\gamma_\cN(\mu)^\prime.$$
In particular, as $f_\lambda\in\dom \gamma_\cN(\mu)^\prime$ there exists an $h\in\dom A_N$ such that 
$$f_\lambda=(A_N-\overline\mu)h\quad\text{and}\quad h\vert_\cC= (\lambda-\overline\mu)^{-1} \psi;$$
cf.\ Lemma~\ref{gammaadj}.
Define $g_{\overline\mu}:=f_\lambda-(\lambda-\overline\mu)h$.
Then $\partial_\nu g_{\overline\mu}\vert_\cC=\partial_\nu f_\lambda\vert_\cC=\varphi$ and
$$(\cL-\overline\mu)g_{\overline\mu}=(\cL -\overline\mu)\bigl(f_\lambda-(\lambda-\overline\mu)h\bigr)=(\lambda-\overline\mu)f_\lambda-(\lambda-\overline\mu)(A_N-\overline\mu)h=0.$$
Moreover, we have $f_\lambda\vert_\cC-g_{\overline\mu}\vert_\cC=(\lambda-\overline\mu)h\vert_\cC=\psi$ and therefore
\[
\{
 \varphi , \psi \}=
 \bigl\{
 \varphi , f_\lambda\vert_\cC-g_{\overline\mu}\vert_\cC\bigr\}
=
\bigl\{
 \partial_\nu f_\lambda\vert_\cC , f_\lambda\vert_\cC \bigr\}
- \bigl\{
   \partial_\nu g_{\overline\mu}\vert_\cC, g_{\overline\mu}\vert_\cC
  \bigr\}\in\cN(\lambda)-\cN(\overline\mu)
\]
as required
\end{proof}

The following corollary is a consequence of Theorem~\ref{dnthm}, Lemma~\ref{gammalambdamu} 
and the fact that 
\[
\dom\bigl(\cD(\lambda)-\cD(\overline\mu)\bigr)=\dom\gamma_\cD(\lambda)\cap\dom\gamma_\cD(\mu)
\]
for all $\mu\in\sigma_p(A_D)$.
Similarly, 
\[
\dom\bigl(\cN(\lambda)-\cN(\overline\mu)\bigr)=\dom\gamma_\cN(\lambda)\cap\dom\gamma_\cN(\mu)
\]
for all $\mu\in\sigma_p(A_N)$.

\begin{corollary}\label{dncor}
Let $\lambda,\mu\in\dC$. Then
\[
 \cD(\lambda)-\cD(\overline\mu)
=(\overline\mu-\lambda)\gamma_\cD(\mu)^\prime\bigl(I+(\lambda-\mu)(A_D-\lambda)^{-1}\bigr)\gamma_\cD(\mu)
\]
and
\[
 \cN(\lambda)-\cN(\overline\mu)
=(\lambda-\overline\mu)\gamma_\cN(\mu)^\prime\bigl(I+(\lambda-\mu)(A_N-\lambda)^{-1}\bigr)\gamma_\cN(\mu).
\]
\end{corollary}

Recall that $\cD(\lambda) \colon H^{1/2}(\cC)\rightarrow H^{-1/2}(\cC)$ and 
$\cN(\lambda) \colon H^{-1/2}(\cC)\rightarrow H^{1/2}(\cC)$ are 
bounded operators for all $\lambda\in\rho(A_D)$ and $\lambda\in\rho(A_N)$, respectively. 
It follows from Corollary~\ref{dncor} that the functions
$\lambda\mapsto \cD(\lambda)$ and $\lambda\mapsto \cN(\lambda)$ are analytic on $\rho(A_D)$ and $\rho(A_N)$, respectively. In the next proposition
we show that under appropriate assumptions this extends also to points in  $\sigma_p(A_D)$ and $\sigma_p(A_N)$.

\begin{proposition}
Let $\lambda_0 \in \dC$. 
Then one has the following.
\begin{itemize}
 \item [{\rm (i)}] For all $\varphi \in \dom \cD(\lambda_0)$ the map $\lambda \mapsto (\cD(\lambda)\varphi,\varphi)_{H^{-1/2}(\cC) \times H^{1/2}(\cC)}$ is 
differentiable at $\lambda_0$ and 
\[
 \frac{d}{d\lambda} (\cD(\lambda)\varphi,\varphi)_{H^{-1/2}(\cC) \times H^{1/2}(\cC)}\Bigl|_{\lambda=\lambda_0}
=-\|f_{\lambda_0}\|_{L^2(\Omega)}^2,
\]
where $f_{\lambda_0} \in (\ker (A_D - \lambda_0))^\perp$ is the unique element such that 
$ \{ \varphi,f_{\lambda_0} \} \in \gamma_\cD(\lambda_0)$.
\item [{\rm (ii)}] For all $\varphi \in \dom \cN(\lambda_0)$ the map
$\lambda \mapsto (\cN(\lambda)\phi,\phi)_{H^{1/2}(\cC) \times H^{-1/2}(\cC)}$ is 
differentiable at $\lambda_0$ and 
\[
 \frac{d}{d\lambda} (\cN(\lambda)\phi,\phi)_{H^{1/2}(\cC) \times H^{-1/2}(\cC)}\Bigl|_{\lambda=\lambda_0}
= \|f_{\lambda_0}\|_{L^2(\Omega)}^2,
\]
where $f_{\lambda_0} \in (\ker (A_N - \lambda_0))^\perp$ is the unique element such that 
$ \{ \varphi,f_{\lambda_0} \} \in \gamma_\cN(\lambda_0)$.
\end{itemize}
\end{proposition}

\begin{proof}
We show Statement (ii).
Let $\lambda_0 \in \dC$, $\varphi \in \dom \cN(\lambda_0)$, and
$P$ be the orthogonal projection in $L^2(\Omega)$ onto 
$(\ker (A_N - \lambda_0))^\perp$.
There exists a unique $f_{\lambda_0} \in H^1(\Omega)$ such that 
$P f_{\lambda_0} = f_{\lambda_0}$, $\cL f_{\lambda_0} = \lambda_0 f_{\lambda_0}$ and 
$ \{ \varphi, f_{\lambda_0}|_\cC \} \in \cN(\lambda_0)$.
Let $\lambda \in \dC \setminus \{ \lambda_0 \} $ and suppose that $|\lambda - \lambda_0|$ 
is small.
Then $\lambda \in \rho(A_N)$ and 
\[
\left\{ \varphi, \frac{1}{\lambda - \lambda_0} (\cN(\lambda) \varphi - f_{\lambda_0}|_\cC) \right\}
\in \frac{1}{\lambda - \lambda_0} \bigl(\cN(\lambda) - \cN(\lambda_0) \bigr)
= \gamma_\cN(\overline \lambda)' \gamma_\cN(\lambda_0)
\]
by Theorem~\ref{dnthm}.
Moreover, $ \{ \varphi, f_{\lambda_0} \} \in \gamma_\cN(\lambda_0)$.
The proof of Theorem~\ref{dnthm} gives that 
\[
\left\{ f_{\lambda_0}, \frac{1}{\lambda - \lambda_0} (\cN(\lambda) \varphi - f_{\lambda_0}|_\cC) \right\}
\in \gamma_\cN(\overline \lambda)'.  
\]
By definition of the adjoint one has 
\begin{equation}\label{elemma1;1}
\begin{split}
\left( \frac{1}{\lambda - \lambda_0} (\cN(\lambda) \varphi - f_{\lambda_0}|_\cC), \varphi\right)_{H^{1/2}(\cC) \times H^{-1/2}(\cC)}
&= (f_{\lambda_0}, \gamma_\cN(\lambda) \varphi)_{L^2(\Omega)}\\
&= (f_{\lambda_0}, P \gamma_\cN(\lambda) \varphi)_{L^2(\Omega)}
\end{split}
\end{equation}
for all $ \{ \varphi, \gamma_\cN(\lambda) \varphi \} \in \gamma_\cN(\lambda)$.
Since $\varphi \in \dom \gamma_\cN(\lambda_0)$ it follows from
Lemma~\ref{gammalambdamu} that 
\[
P \gamma_\cN(\lambda) \varphi
= f_{\lambda_0} + (\lambda - \lambda_0) (A_N - \lambda)^{-1} f_{\lambda_0} .
\]
Let $(e_k)_{k \in \dN}$ be an orthonormal basis in $L^2(\Omega)$
of eigenfunctions for $A_N$.
Suppose that $A_N e_k = \mu_k e_k$ for all $k \in \dN$.
Then 
\[
(\lambda - \lambda_0) (A_N - \lambda)^{-1} f_{\lambda_0}
= \sum_{\scriptstyle k \in \dN \atop
        \scriptstyle \mu_k \neq \lambda_0}
    \frac{\lambda - \lambda_0}{\mu_k - \lambda} (f_{\lambda_0}, e_k)_{L^2(\Omega)} e_k .
\]
So $\lim_{\lambda \to \lambda_0} (\lambda - \lambda_0) (A_N - \lambda)^{-1} f_{\lambda_0} = 0$
in $L^2(\Omega)$
and it follows from (\ref{elemma1;1}) that 
\[
\lim_{\lambda \to \lambda_0} 
\left( \frac{1}{\lambda - \lambda_0} (\cN(\lambda) \varphi - f_{\lambda_0}|_\cC), \varphi\right)_{H^{1/2}(\cC) \times H^{-1/2}(\cC)}
= \|f_{\lambda_0}\|_{L^2(\Omega)}^2 .  \]
Hence $\lambda \mapsto (\cN(\lambda)\phi,\phi)_{H^{1/2}(\cC) \times H^{-1/2}(\cC)}$ is 
differentiable at $\lambda_0$ with derivative 
$\|f_{\lambda_0}\|_{L^2(\Omega)}^2$.
\end{proof}

In the next theorem we show how $A_D$ and $A_N$ are related to each other in a Krein type resolvent formula.
For the case that $\lambda\in\dC$ belongs to
the resolvent set of both operators $A_D$ and $A_N$ such formulae are well known and can be 
found in e.g.\ \cite{AlpayB,BehL1,BehL2,BGW,GesM2,GesM3,Mal,PR,Post}.
However, our aim is to show that the correspondence between $(A_D-\lambda)^{-1}$ and 
$(A_N-\lambda)^{-1}$ in terms of $\gamma_\cD(\lambda)$, $\gamma_\cN(\lambda)$,
and the Dirichlet-to-Neumann map $\cD(\lambda)$ and Neumann-to-Dirichlet map $\cN(\lambda)$ is also valid if $\lambda$ is an eigenvalue 
of one or both of the operators $A_D$ and $A_N$.

\begin{theorem}\label{resdiffthm}
If $\lambda\in\dC$ then
\[
  (A_N-\lambda)^{-1}-(A_D-\lambda)^{-1}=\gamma_\cD(\lambda)\cN(\lambda)\gamma_\cD(\overline\lambda)^\prime
=\gamma_\cN(\lambda)\cD(\lambda)\gamma_\cN(\overline\lambda)^\prime.
\]
\end{theorem}

\begin{proof}
We verify the formula
\begin{equation}\label{eadanform}
 (A_N-\lambda)^{-1}-(A_D-\lambda)^{-1}=\gamma_\cN(\lambda)\cD(\lambda)\gamma_\cN(\overline\lambda)^\prime;
\end{equation}
the proof of the corresponding formula with $\gamma_\cD(\lambda)\cN(\lambda)\gamma_\cD(\overline\lambda)^\prime$
on the right hand side 
is very similar.

For the inclusion $\subset$ in \eqref{eadanform} let 
$h,h_N,h_D \in L^2(\Omega)$, suppose that $\{h,h_N\}\in (A_N-\lambda)^{-1}$ and $\{h,h_D\}\in (A_D-\lambda)^{-1}$, so that
$$\{h,h_N-h_D\}\in (A_N-\lambda)^{-1}-(A_D-\lambda)^{-1}.$$
Then
\begin{equation}\label{ehnh}
 (A_N-\lambda)h_N=h\qquad\text{,}\qquad (A_D-\lambda)h_D=h,
\end{equation}
and it follows from Lemma~\ref{gammaadj} that 
\begin{equation}\label{1}
\{h,h_N\vert_\cC\}\in \gamma_\cN(\overline\lambda)^\prime.
\end{equation}
Let us show that
$h_N\vert_\cC\in\dom \cD(\lambda)$. This is clear if $\lambda\not\in\sigma_p(A_D)$. Assume that $\lambda\in\sigma_p(A_D)$.
Then for all $f_\lambda\in\ker(A_D-\lambda)$ one deduces from Green's identity, $f_\lambda\vert_\cC=0$ and \eqref{ehnh} that 
\begin{eqnarray}\label{ahagut}
\lefteqn{
(h_N\vert_\cC,\partial_\nu f_\lambda\vert_\cC)_{H^{1/2}(\cC) \times H^{-1/2}(\cC)}
} \hspace{10mm} \\*
&=& (h_N\vert_\cC,\partial_\nu f_\lambda\vert_\cC)_{H^{1/2}(\cC) \times H^{-1/2}(\cC)} 
   - (\partial_\nu h_N\vert_\cC, f_\lambda\vert_\cC)_{H^{-1/2}(\cC) \times H^{1/2}(\cC)}  \nonumber\\
&=& (\cL h_N,f_\lambda)_{L^2(\Omega)}
   -(h_N,\cL f_\lambda)_{L^2(\Omega)}= ((A_N-\lambda)h_N,f_\lambda)_{L^2(\Omega)} \nonumber\\
&=& (h,f_\lambda)_{L^2(\Omega)}
=((A_D-\lambda)h_D,f_\lambda)_{L^2(\Omega)}
=(h,(A_D-\lambda)f_\lambda)_{L^2(\Omega)}=0, \nonumber
\end{eqnarray}
and hence $h_N\vert_\cC\in\dom \cD(\lambda)$ by Proposition~\ref{domdomprop}.
Thus there exists a $k_\lambda\in H^1(\Omega)$
such that $\cL k_\lambda=\lambda k_\lambda$, $\{k_\lambda\vert_\cC,\partial_\nu k_\lambda\vert_\cC \}\in\cD(\lambda)$ and $k_\lambda\vert_\cC=h_N\vert_\cC$. 
Observe that $k_\lambda:= h_N-h_D$ is a possible choice.
In fact, $h_N-h_D\in H^1(\Omega)$  as $h_N\in \dom A_N$ and $h_D\in\dom A_D$, 
and $\cL(h_N-h_D)=\lambda(h_N-h_D)$ follows from \eqref{ehnh}.
Moreover, we have $(h_N-h_D)\vert_\cC=h_N\vert_\cC$ and $\partial_\nu (h_N-h_D)\vert_\cC=- \partial_\nu h_D\vert_\cC$.
It follows that
\begin{equation}\label{2}
\bigl\{h_N\vert_\cC,-\partial_\nu h_D\vert_\cC\bigr\}=\bigl\{(h_N-h_D)\vert_\cC,\partial_\nu (h_N-h_D)\vert_\cC \bigr\}\in\cD(\lambda).
\end{equation}
Next we show that $-\partial_\nu h_D\vert_\cC\in\dom\gamma_\cN(\lambda)$. This is clear if $\lambda\not\in\sigma_p(A_D)$. Assume now that $\lambda\in\sigma_p(A_D)$
and let $g_\lambda\in\ker(A_N-\lambda)$.
Then we compute
in a similar way as in \eqref{ahagut} that
\begin{eqnarray*}
\lefteqn{
(-\partial_\nu h_D\vert_\cC,g_\lambda\vert_\cC)_{H^{-1/2}(\cC) \times H^{1/2}(\cC)}
} \hspace{10mm} \\*
&=&(h_D\vert_\cC,\partial_\nu g_\lambda\vert_\cC)_{H^{1/2}(\cC) \times H^{-1/2}(\cC)}
    -(\partial_\nu h_D\vert_\cC,g_\lambda\vert_\cC)_{H^{-1/2}(\cC) \times H^{1/2}(\cC)} \\
&=&(\cL h_D,g_\lambda)_{L^2(\Omega)}
   -(h_D,\cL g_\lambda)_{L^2(\Omega)}
=((A_D-\lambda)h_D,g_\lambda)_{L^2(\Omega)}\\
&=&(h,g_\lambda)_{L^2(\Omega)}
=((A_N-\lambda)h_N,g_\lambda)_{L^2(\Omega)}
=(h,(A_N-\lambda)g_\lambda)_{L^2(\Omega)}=0.
\end{eqnarray*}
Therefore $-\partial_\nu h_D\vert_\cC\in\dom\gamma_\cN(\lambda)$ follows from \eqref{edomdom2} and Proposition~\ref{domdomprop}.
This implies that
\begin{equation}\label{3}
 \{-\partial_\nu h_D\vert_\cC,h_N-h_D\}\in\gamma_\cN(\lambda).
\end{equation}
From \eqref{1}, \eqref{2}, and \eqref{3} we now conclude that
$\{h,h_N-h_D\}\in\gamma_\cN(\lambda)\cD(\lambda)\gamma_\cN(\overline\lambda)^\prime$
which shows the inclusion $\subset$ in \eqref{eadanform}.

We now prove the inclusion $\supset$ in \eqref{eadanform}.
Let $\{h,k_\lambda\}\in \gamma_\cN(\lambda)\cD(\lambda)\gamma_\cN(\overline\lambda)^\prime$.
Then there exists an $h_N\in\dom A_N$ such that $h=(A_N-\lambda)h_N$
and $\{h,h_N\vert_\cC\}\in \gamma_\cN(\overline\lambda)^\prime$.
Moreover, $\cL k_\lambda=\lambda k_\lambda$, $k_\lambda\in H^1(\Omega)$, $k_\lambda\vert_\cC=h_N\vert_\cC$
and $\{k_\lambda\vert_\cC,\partial_\nu k_\lambda\vert_\cC\}\in\cD(\lambda)$
and $\{\partial_\nu k_\lambda\vert_\cC,k_\lambda\}\in\gamma_\cN(\lambda)$.
It is clear that $\{h,h_N\}\in (A_N-\lambda)^{-1}$.
Let 
\begin{equation}\label{hdhn}
h_D:=h_N-k_\lambda.
\end{equation}
Then we have $h_D\in H^1(\Omega)$ and $h_D\vert_\cC=h_N\vert_\cC-k_\lambda\vert_\cC=0$.
Moreover, as 
$$(\cL-\lambda)h_D=(\cL-\lambda)(h_N-k_\lambda)=(\cL-\lambda)h_N=(A_N-\lambda)h_N=h$$
it follows that $h_D\in\dom A_D$ and $(A_D-\lambda)h_D=h$.
This implies  $\{h,h_D\}\in (A_D-\lambda)^{-1}$ and from \eqref{hdhn} we conclude that 
\begin{equation*}
\{h,k_\lambda\} =\{h,h_N-h_D\} =\{h,h_N\}-\{h,h_D\}\in (A_N-\lambda)^{-1}- (A_D-\lambda)^{-1}.
\end{equation*}
This shows the inclusion $\supset$ in \eqref{eadanform}.
Theorem~\ref{resdiffthm} is proved.
\end{proof}

\section{Dirichlet-to-Neumann and Neumann-to-Dirichlet maps in $L^2(\cC)$}\label{sec5}

In this section we consider the restrictions
\[
\begin{split}
D(\lambda)&=\bigl\{\{f_\lambda\vert_\cC,\partial_\nu f_\lambda\vert_\cC\}\in\cD(\lambda):
 f_\lambda \in H^1(\Omega), \; \cL f_\lambda = \lambda f_\lambda \mbox{ and } 
   \partial_\nu f_\lambda\vert_\cC\in L^2(\cC) \bigr\},\\
N(\lambda)&=\bigl\{\{\partial_\nu f_\lambda\vert_\cC,f_\lambda\vert_\cC\}\in\cN(\lambda):
 f_\lambda \in H^1(\Omega), \; \cL f_\lambda = \lambda f_\lambda \mbox{ and } 
   \partial_\nu f_\lambda\vert_\cC\in L^2(\cC) \bigr\},
\end{split}
\]
of the Dirichlet-to-Neumann and Neumann-to-Dirichlet map in $L^2(\cC)$. 
Since the trace $f_\lambda\vert_\cC$ of a function $f_\lambda\in H^1(\Omega)$ belongs to
$H^{1/2}(\cC)\subset L^2(\cC)$ the relations $D(\lambda)$ and $N(\lambda)$ are contained 
in $L^2(\cC)\times L^2(\cC)$. 
Clearly, 
$$
D(\lambda)=\cD(\lambda)\cap \bigl(L^2(\cC)\times L^2(\cC)\bigr)\quad\text{and}\quad N(\lambda)=\cN(\lambda)\cap \bigl(L^2(\cC)\times L^2(\cC)\bigr),
$$
and, in particular, 
$D(\lambda) \subset \cD(\lambda)$ and $N(\lambda) \subset \cN(\lambda)$.

In the next theorem the domains, kernels and multivalued parts of $D(\lambda)$
and $N(\lambda)$ are specified.
It is remarkable that $\mul D(\lambda)$ and $\ker N(\lambda)$ coincide with $\mul\cD(\lambda)$ and $\ker\cN(\lambda)$,
respectively.
These facts and the assertions on the domains below are essentially consequences of the regularity results 
\begin{equation*}
 \dom A_D\subset H^{3/2}(\Omega)\quad\text{and}\quad\dom A_N\subset H^{3/2}(\Omega)
\end{equation*}
due to Jerison and Kenig 
\cite{JK2,JK}, and Gesztesy and Mitrea \cite{GesM2,GesM3}. The following lemma is particularly useful; cf.\ \cite{GesM2}, Lemma 2.3 and Lemma 2.4.

\begin{lemma} \label{lDNmaps401}
The following assertions are valid.
\begin{itemize}
 \item [{\rm (i)}] Let $f\in H^{3/2}(\Omega)$ and suppose that 
$\cL f\in L^2(\Omega)$. Then $f\vert_\cC \in H^1(\cC)$ and
$\partial_\nu f\vert_\cC \in L^2(\cC)$.
\item [{\rm (ii)}] For all $\varphi\in H^1(\cC)$ there exists a $g\in H^{3/2}(\Omega)$ such that $\cL g\in L^2(\Omega)$ and $g\vert_\cC=\varphi$.
\item [{\rm (iii)}] For all $\psi\in L^2(\cC)$ there exists an $h\in H^{3/2}(\Omega)$ such that $\cL h\in L^2(\Omega)$ and $\partial_\nu h\vert_\cC=\psi$.
\end{itemize}
\end{lemma}

\begin{theorem}\label{thmdn}
Let $\lambda\in\dC$. 
 The domains of the Dirichlet-to-Neumann map $D(\lambda)$ 
and Neumann-to-Dirichlet map $N(\lambda)$ in $L^2(\cC)$ are  
 \begin{equation}\label{eabcdef}
 \dom D(\lambda)=\bigl\{\varphi\in H^1(\cC): 
    (\varphi,\partial_\nu f_\lambda\vert_\cC)_{L^2(\cC)}=0 \mbox{ for all } 
          f_\lambda\in\ker(A_D-\lambda)\bigr\}
\end{equation}
and
\begin{equation*}
 \dom N(\lambda)=\bigl\{\psi\in L^2(\cC): (\psi,f_\lambda\vert_\cC)_{L^2(\cC)} = 0 \mbox{ for all } 
            f_\lambda\in\ker(A_N-\lambda)\bigr\}.
 \end{equation*}
Moreover, 
\begin{itemize}
  \item [{\rm (i)}] $\ker D(\lambda)=\ker\cD(\lambda)\subset H^1(\cC)$,
  \item [{\rm (ii)}] $\mul D(\lambda)=\mul\cD(\lambda)\subset L^2(\cC)$,
  \end{itemize}
  and,
  \begin{itemize}
  \item [{\rm (iii)}] $\ker N(\lambda)=\ker\cN(\lambda)\subset L^2(\cC)$,
  \item [{\rm (iv)}] $\mul N(\lambda)=\mul\cN(\lambda)\subset H^1(\cC)$.
  \end{itemize}
\end{theorem}

\begin{proof}
We verify the assertions for $D(\lambda)$.
Recall first that $\dom\cD(\lambda)$ is given by \eqref{edomgammad}.
Hence the inclusion $\subset$ in \eqref{eabcdef} for 
$\dom D(\lambda)$
follows if we show that for all $f_\lambda\in H^1(\Omega)$ such that 
$\cL f_\lambda=\lambda f_\lambda$ and $\partial_\nu f_\lambda\vert_\cC\in L^2(\cC)$
it follows that $f_\lambda\vert_\cC\in H^1(\cC)$.
By Lemma~\ref{lDNmaps401}~(iii) there exists a $g\in H^{3/2}(\Omega)$ such that
\begin{equation*}
\cL g\in L^2(\Omega)\quad\text{and}\quad \partial_\nu g\vert_\cC=\partial_\nu f_\lambda\vert_\cC. 
\end{equation*}
Then $g-f_\lambda\in H^1(\Omega)$, $\cL(g-f_\lambda)\in L^2(\Omega)$ and $\partial_\nu(g-f_\lambda)\vert_\cC=0$, that is,
$g-f_\lambda\in\dom A_N$.
Hence $g-f_\lambda\in H^{3/2}(\Omega)$ by \cite{GesM2}, Theorem 2.6 and Lemma 4.8. 
As $g\in H^{3/2}(\Omega)$ this yields $f_\lambda\in H^{3/2}(\Omega)$ and therefore Lemma~\ref{lDNmaps401}~(i)
implies $f_\lambda\vert_\cC\in H^1(\cC)$. 
For the inclusion $\supset$ in \eqref{eabcdef} let $\varphi\in H^1(\cC)$ and assume that $(\varphi,\partial_\nu f_\lambda\vert_\cC)=0$
for all $f_\lambda\in\ker(A_D-\lambda)$.
It follows from \eqref{edomgammad} that $\varphi\in\dom\cD(\lambda)$.
Hence there exists an
$f_\lambda\in H^1(\Omega)$ such that $\cL f_\lambda=\lambda f_\lambda$ and $f_\lambda\vert_\cC=\varphi$. 
By Lemma~\ref{lDNmaps401}~(ii)
there exists a $g\in H^{3/2}(\Omega)$ such that 
\begin{equation*}
\cL g\in L^2(\Omega)\quad\text{and}\quad  g\vert_\cC= f_\lambda\vert_\cC. 
\end{equation*}
It follows that 
$g-f_\lambda\in H^1(\Omega)$, $\cL(g-f_\lambda)\in L^2(\Omega)$ and $(g-f_\lambda)\vert_\cC=0$, that is,
$g-f_\lambda\in\dom A_D$.
Hence $g-f_\lambda\in H^{3/2}(\Omega)$ by \cite{GesM2}, Lemma 3.4
As $g\in H^{3/2}(\Omega)$ this yields 
$f_\lambda\in H^{3/2}(\Omega)$ and therefore Lemma~\ref{lDNmaps401}~(i)
implies $\partial_\nu f_\lambda\vert_\cC\in L^2(\cC)$.
We have shown $\{\varphi,\partial_\nu f_\lambda\vert_\cC\}=\{f_\lambda\vert_\cC,\partial_\nu f_\lambda\vert_\cC\}\in D(\lambda)$
and, in particular, $\varphi\in\dom D(\lambda)$.
The assertion on $\dom D(\lambda)$ in \eqref{eabcdef} is shown.

Next we prove (i) and (ii).
As $D(\lambda)$ is contained in $\cD(\lambda)$ it is clear that 
$\ker D(\lambda)\subset\ker\cD(\lambda)$ and $\mul D(\lambda)\subset\mul\cD(\lambda)$.
In order to prove the inclusion $\ker D(\lambda)\supset\ker \cD(\lambda)$ in (i), let 
$f_\lambda\vert_\cC\in\ker\cD(\lambda)$.
Then
$\{f_\lambda\vert_\cC,0\}\in\cD(\lambda)$ and it follows from the definition that $\{f_\lambda\vert_\cC,0\}\in D(\lambda)$.
This shows 
$f_\lambda\vert_\cC\in\ker D(\lambda)$ and (i) is proven.
For (ii) it remains to show the inclusion $\mul D(\lambda)\supset
\mul \cD(\lambda)$.
For this let $\psi\in\mul\cD(\lambda)$.
Then $\{0,\psi\}\in\cD(\lambda)$
and hence there exists an $f_\lambda\in H^1(\Omega)$ such that 
$\cL f_\lambda=\lambda f_\lambda$, $f_\lambda \vert_\cC=0$ and $\partial_\nu f_\lambda\vert_\cC=\psi$. 
This implies $f_\lambda\in\dom A_D$
and from \cite{GesM2}, Lemma 3.4, we conclude that $f_\lambda\in H^{3/2}(\Omega)$.
But then $\psi=\partial_\nu f_\lambda\vert_\cC\in L^2(\partial\Omega)$ 
by Lemma~\ref{lDNmaps401}~(i)
and therefore 
$\{0,\partial_\nu f_\lambda\vert_\cC\}=\{0,\psi\}\in D(\lambda)$, that is, $\psi\in\mul D(\lambda)$.
\end{proof}

As an immediate consequence of \eqref{emulker} and Theorem~\ref{thmdn} we obtain
\begin{equation}\label{mulker2}
 \ker D(\lambda)=\mul N(\lambda)=\bigl\{f_\lambda\vert_\cC:f_\lambda\in \ker (A_N-\lambda) \bigr\}
\end{equation}
and
\begin{equation}\label{mulker3}
 \ker N(\lambda)=\mul D(\lambda)=\bigl\{\partial_\nu f_\lambda\vert_\cC:f_\lambda\in \ker (A_D-\lambda) \bigr\}.
\end{equation}

Furthermore, as a consequence of 
Lemma~\ref{mullem} we obtain the following corollary.
Item (i) coincides with \cite{AEKS}, Proposition 4.11.

\begin{corollary}\label{mulcor}
Let $\lambda\in\dC$.
Then
\begin{itemize}
 \item [(i)] $\mul D(\lambda)=\{0\}$ if and only if $\lambda\not\in\sigma_p(A_D)$,
 \item [(ii)] $\mul N(\lambda)=\{0\}$ if and only if $\lambda\not\in\sigma_p(A_N)$,
\end{itemize} 
and,
\begin{itemize}
 \item [(iii)] $\ker N(\lambda)\not=\{0\}$ if and only if $\lambda\in\sigma_p(A_D)$,
 \item [(iv)] $\ker D(\lambda)\not=\{0\}$ if and only if $\lambda\in\sigma_p(A_N)$.
\end{itemize} 
\end{corollary}

In the following we investigate the Neumann-to-Dirichlet map in  $L^2(\cC)$.  
We will also make use of the restriction $\gamma_N(\lambda)$ of 
$\gamma_\cN(\lambda)$ onto $L^2(\cC)$ given by
$$
\gamma_N(\lambda):=\bigl\{\{\partial_\nu f_\lambda\vert_\cC,f_\lambda \} \in L^2(\cC)\times L^2(\Omega):
  f_\lambda\in H^1(\Omega), \; \cL f_\lambda=\lambda f_\lambda
   \mbox{ and } \partial_\nu f_\lambda\vert_\cC\in L^2(\cC)\bigr\},
$$
which is now regarded as an operator or relation in $L^2(\cC)\times L^2(\Omega)$.
It is important to note that
\begin{equation*}
\dom\gamma_N(\lambda)=\dom N(\lambda)\quad\text{and}\quad \mul\gamma_N(\lambda)=\ker(A_N-\lambda),
\end{equation*}
and, in particular, $\dom\gamma_N(\lambda)=L^2(\cC)$ if and only if $\lambda\not\in\sigma_p(A_N)$.
Note that an analogous $L^2$-restriction of 
$\gamma_\cD(\lambda)$ does not lead to a smaller operator or relation.
However, for consistency we shall write here $\gamma_D(\lambda)$
instead of $\gamma_\cD(\lambda)$, that is, 
$$
\gamma_D(\lambda):=\bigl\{\{f_\lambda\vert_\cC,f_\lambda \} \in L^2(\cC)\times L^2(\Omega):
   f_\lambda\in H^1(\Omega) \mbox{ and } \cL f_\lambda=\lambda f_\lambda\bigr\},
$$
and  $\gamma_D(\lambda)$ is regarded as an operator or relation in $L^2(\cC)\times L^2(\Omega)$.
Obviously we have
\begin{equation*}
\dom\gamma_D(\lambda)=\dom D(\lambda)\quad\text{and}\quad \mul\gamma_D(\lambda)=\ker(A_D-\lambda).
\end{equation*}
The other statements and formulas for $\gamma_\cD(\lambda)$ and $\gamma_\cN(\lambda)$
in the previous section remain true for $\gamma_D(\lambda)$ and $\gamma_N(\lambda)$ in an appropriate form.
In particular,  $\gamma_\cD(\lambda)^\prime$ and 
$\gamma_\cN(\lambda)^\prime$ in Lemma~\ref{gammaadj} can now be regarded as operators or relations $\gamma_\cD(\lambda)^*$ and $\gamma_N(\lambda)^*$,
respectively, in $L^2(\Omega)\times L^2(\cC)$.
Specifically,  if $\lambda\in\dC$ then
\begin{equation*}
  \gamma_D(\lambda)^* =\bigl\{\{(A_D-\overline\lambda)g, -\partial_\nu g\vert_\cC\}:g\in\dom A_D\bigr\}
\end{equation*}
and
\begin{equation*}
 \gamma_N(\lambda)^* =\bigl\{\{(A_N-\overline\lambda)g, g\vert_\cC\}:g\in\dom A_N\bigr\}.
\end{equation*}

We list some useful consequences in the next corollary.

\begin{corollary}\label{gammaadjcor2}
The following assertions are valid.
\begin{itemize}
 \item [{\rm (i)}] If $\lambda\in\rho(A_D)$ then
\[
 \gamma_D(\lambda)^*  \colon  L^2(\Omega)\rightarrow L^2(\cC),\quad 
     h\mapsto -\partial_\nu\bigl((A_D-\overline\lambda)^{-1}h\bigr)\vert_\cC,
\]
is bounded.
Moreover, $\gamma_D(\lambda) \colon L^2(\cC)\supset \dom\gamma_D(\lambda) \rightarrow L^2(\Omega)$ is a bounded operator 
with dense domain $\dom\gamma_D(\lambda)=H^1(\cC)$
and $\gamma_D(\lambda)$ admits a unique continuous extension from $L^2(\cC)$ into $L^2(\Omega)$.
 \item [{\rm (ii)}] If $\lambda\in\rho(A_N)$ then 
\[
 \gamma_N(\lambda)^*  \colon  L^2(\Omega)\rightarrow L^2(\cC),\quad 
   h\mapsto \bigl((A_N-\overline\lambda)^{-1}h\bigr)\vert_\cC
\]
 and  $\gamma_N(\lambda) \colon L^2(\cC)\rightarrow L^2(\Omega)$ are compact operators.
\end{itemize}
\end{corollary}

\begin{proof}
It is clear that for all $\lambda\in\rho(A_D)$ (or $\lambda\in\rho(A_N)$) the operator 
$\gamma_D(\lambda)^*$ (or $\gamma_N(\lambda)^*$, respectively) is closed and defined on 
the whole space $L^2(\Omega)$, and hence bounded by the closed graph theorem.
Thus $\gamma_D(\lambda)^{**}$ is bounded as well and this implies that 
$\gamma_D(\lambda)$ admits a unique continuous extension on $L^2(\cC)$ which is the closure 
$\overline{\gamma_D(\lambda)}=\gamma_D(\lambda)^{**}$.
The operator $\gamma_N(\lambda)$
is defined on $L^2(\cC)$ and coincides with $\gamma_N(\lambda)^{**}$, and hence it is bounded.
In particular, $\gamma_N(\lambda)$ is closed as an operator
from $L^2(\cC)$ into $L^2(\Omega)$, and therefore it also closed as an operator from $L^2(\cC)$ into $H^1(\Omega)$.
As $H^1(\Omega)$ is compactly
embedded in $L^2(\Omega)$ this implies that $\gamma_N(\lambda)$ and consequently
also $\gamma_N(\lambda)^*$ are compact.
\end{proof}

Theorem~\ref{dnthm} and Corollary~\ref{dncor} have the following analogue statements
for $D(\lambda)$ and $N(\lambda)$.

\begin{corollary}
Let $\lambda,\mu\in\dC$. Then
\[
\begin{split}
 D(\lambda)-D(\overline\mu)&=(\overline\mu-\lambda)\gamma_D(\mu)^*\gamma_D(\lambda),\\
 N(\lambda)-N(\overline\mu)&=(\lambda-\overline\mu)\gamma_N(\mu)^*\gamma_N(\lambda),
\end{split}
\]
and, in particular,
\begin{equation}\label{useful}
\begin{split}
 D(\lambda)-D(\overline\mu)&=(\overline\mu-\lambda)\gamma_D(\mu)^*\bigl(I+(\lambda-\mu)(A_D-\lambda)^{-1}\bigr)\gamma_D(\mu),\\
 N(\lambda)-N(\overline\mu)&=(\lambda-\overline\mu)\gamma_N(\mu)^*\bigl(I+(\lambda-\mu)(A_N-\lambda)^{-1}\bigr)\gamma_N(\mu).
\end{split}
\end{equation}
\end{corollary}

In the next theorem we show that for all $\lambda\in\dR$ the Neumann-to-Dirichlet map is selfadjoint in $L^2(\cC)$.
Its operator part is compact and has at most finitely many negative eigenvalues.
In particular, the Neumann-to-Dirichlet map 
is bounded from below for all $\lambda\in\dR$. 
For a selfadjoint operator or relation $S$
we shall denote by $\kappa_-(S)$ the number of strictly negative eigenvalues,
counted with multiplicity. Similarly we denote by $\kappa_+(S)$
the number of strictly positive eigenvalues of $S$, counted with multiplicity.

\begin{theorem}\label{dnmapsa}
Let $\lambda\in\dR$.
Then the Neumann-to-Dirichlet map $N(\lambda)$ is a selfadjoint relation in $L^2(\cC)$ defined on the closed subspace
\begin{equation*}
\dom N(\lambda)= \bigl\{\psi\in L^2(\cC):(\psi,f_\lambda\vert_\cC)_{L^2(\cC)}=0 \mbox{ for all }
f_\lambda\in \ker (A_N-\lambda)\bigr\}\subset L^2(\cC)
\end{equation*}
with multivalued part $\mul N(\lambda)= \{f_\lambda\vert_\cC:f_\lambda\in \ker (A_N-\lambda)\}$.
The operator part $N_{\rm op}(\lambda)$ of $N(\lambda)$
is a compact selfadjoint operator in the Hilbert space $\dom N(\lambda)$.
Moreover,
\begin{itemize}
 \item [{\rm (i)}] $\kappa_-(N(\lambda))\leq \kappa_-(A_N-\lambda)<\infty$ and $\kappa_+(N(\lambda))=\infty$,
 \item [{\rm (ii)}] $\dim\ker(N(\lambda))=\dim\ker(A_D-\lambda)<\infty$,
\end{itemize}
and
\begin{itemize}
 \item [{\rm (iii)}]$\dim\mul(N(\lambda))=\dim\ker(A_N-\lambda)<\infty$.
\end{itemize}
\end{theorem}

\begin{proof}
The assertions on the domain and multivalued part of $N(\lambda)$ were shown in Theorem~\ref{thmdn}.
The remaining statements 
will be shown in separate steps. 
\vskip 0.2cm
\noindent
{\it Step 1.}
Note first that for all $\mu\in\dR\cap\rho(A_N)$ the Neumann-to-Dirichlet map is an operator with $\dom N(\mu)=L^2(\cC)$ and that
\begin{equation*}
(N(\mu)\varphi,\psi)_{L^2(\Omega)}-(\varphi,N(\mu)\psi)_{L^2(\Omega)}
= (f_\mu\vert_\cC,\partial_\nu g_\mu\vert_\cC)_{L^2(\cC)}
   - (\partial_\nu f_\mu\vert_\cC,g_\mu\vert_\cC)_{L^2(\cC)}
=0
\end{equation*}
by Lemma~\ref{greenlem}, where $f_\mu,g_\mu$ are the unique $H^1$-solutions of $\cL u=\mu u$ such that 
$\partial_\nu f_\mu\vert_\cC=\varphi$ and $\partial_\nu g_\mu\vert_\cC=\psi$.
Therefore $N(\mu)$ is a bounded selfadjoint 
operator in $L^2(\cC)$ for all $\mu\in\dR\cap\rho(A_N)$.
In particular, $N(\mu)$ is closed as an operator in $L^2(\cC)$
and as $\ran N(\mu)\subset H^1(\cC)$ it follows that $N(\mu)$ is also closed as an operator from 
$L^2(\cC)$ into $H^1(\cC)$ 
Hence it is bounded from $L^2(\cC)$ into $H^1(\cC)$.
Since $H^1(\cC)$ is compactly embedded in $L^2(\cC)$ we conclude 
that $N(\mu)$ is a compact selfadjoint 
operator in $L^2(\cC)$ for all $\mu\in\dR\cap\rho(A_N)$. 
Moreover, if $\mu<\essinf\,V$ then $\mu\in\rho(A_N)$ and \eqref{ederi} yields
\begin{equation*}
 \begin{split}
  (N(\mu)\varphi,\varphi)_{L^2(\cC)}&=(f_\mu\vert_\cC,\partial_\nu f_\mu\vert_\cC)_{L^2(\cC)}=
  (f_\mu,\Delta f_\mu)_{L^2(\Omega)}+(\nabla f_\mu,\nabla f_\mu)_{L^2(\Omega)^n}\\
  & \geq (f_\mu,(V-\mu)f_\mu)_{L^2(\Omega)}\geq 0,
 \end{split}
\end{equation*}
that is, $N(\mu)$ is a positive compact operator in $L^2(\cC)$. 
\vskip 0.2cm
\noindent
{\it Step 2.}
In order to show the remaining statements for $N(\lambda)$ and its operator part $N_{\rm op}(\lambda)$ 
we make use of \eqref{useful}.
Fix $\mu<\essinf\, V\leq \min\,\sigma(A_N)$.
Then $\mu\in\dR\cap\rho(A_N)$ and \eqref{useful} implies that
\begin{equation}\label{npertu}
N(\lambda)=K + (\lambda-\mu)^2 \gamma_N(\mu)^*(A_N-\lambda)^{-1}\gamma_N(\mu),
\end{equation}
where we have set
\begin{equation}\label{k}
 K:=N(\mu)+(\lambda-\mu)\gamma_N(\mu)^*\gamma_N(\mu).
\end{equation}
We have shown in Step 1 that $N(\mu)$ is a positive compact operator in $L^2(\cC)$ 
and the same is true for the second summand in \eqref{k}.
In fact, according to Corollary~\ref{gammaadjcor2} both operators
$\gamma_N(\mu)$ and $\gamma_N(\mu)^*$ are compact and hence 
$(\lambda-\mu)\gamma_N(\mu)^*\gamma_N(\mu)$ is a compact positive operator in $L^2(\cC)$.
Thus $K$ in \eqref{k} is a compact positive operator in $L^2(\cC)$.

\vskip 0.2cm
\noindent
{\it Step 3.}
Let $\lambda\in\sigma_p(A_N)$.
In this step we show that $N(\lambda)$ is a selfadjoint relation in $L^2(\cC)$.
By \eqref{npertu} 
it is sufficient to check that the relation
\begin{equation}\label{trel}
T:=\gamma_N(\mu)^*\bigl(A_N-\lambda\bigr)^{-1}\gamma_N(\mu)
\end{equation}
is selfadjoint in $L^2(\cC)$.
We aim to apply Proposition~\ref{ttprop}.
The assumptions in Proposition~\ref{ttprop} are satisfied 
since $A_N-\lambda$ is selfadjoint and $\ran(A_N-\lambda)$ is closed 
because $\lambda$ is an eigenvalue of finite multiplicity, 
$\gamma_N(\mu)$ is a bounded operator from $L^2(\cC)$ into $L^2(\Omega)$ and for all 
$h_\lambda\in \ker(A_N-\lambda)$ we have
$$
\gamma(\mu)^*h_\lambda=(\lambda-\mu)^{-1}h_\lambda\vert_\cC,
$$
so that $\gamma(\mu)^*h_\lambda=0$ implies $h_\lambda\vert_\cC=\partial_\nu h_\lambda\vert_\cC=0$ 
and hence $h_\lambda=0$ by unique continuation. 
Therefore $\gamma(\mu)^*\upharpoonright\ker (A_N-\lambda)$ is boundedly invertible and 
Proposition~\ref{ttprop} yields that the relation $T$
is selfadjoint in $L^2(\cC)$. It follows that $N(\lambda)=N(\lambda)^*$.

\vskip 0.2cm
\noindent
{\it Step 4.} Denote by $\{\lambda_k\}_{k \in \dN}$ the eigenvalues of $A_N$ with 
multiplicities taken into account and ordered in an increasing way. 
For all $\lambda\in(\mu,\infty)$ the eigenvalues of the selfadjoint relation $(A_N-\lambda)^{-1}$ are 
given by $\{(\lambda_k-\lambda)^{-1}: k \in \dN \mbox{ and } \lambda_k\not=\lambda\}$ 
and $\mul(A_N-\lambda)^{-1}=\ker(A_N-\lambda)$.
In particular, 
there are at most finitely many negative eigenvalues $(\lambda_i-\lambda)^{-1}$
with $\lambda_i <\lambda$ of $(A_N-\lambda)^{-1}$ and the positive eigenvalues 
$(\lambda_j-\lambda)^{-1}$ with $\lambda_j >\lambda$ of $(A_N-\lambda)^{-1}$
accumulate to $0$.
Hence the selfadjoint operator part $((A_N-\lambda)^{-1})_{\rm op}$ of $(A_N-\lambda)^{-1}$ 
acting in the Hilbert space $\ran(A_N-\lambda)$ is compact. 
It is not difficult to see that the operator part $T_{\rm op}$ of the selfadjoint relation $T$ in \eqref{trel} is given by
\[
T_{\rm op}=\gamma_N(\mu)^*\bigl((A_N-\lambda)^{-1}\bigr)_{\rm op}\gamma_N(\mu).
\]
It then follows that $T_{\rm op}$ is compact, that $T$ has finitely many negative eigenvalues and 
$\kappa_-(T)\leq \kappa_-(A_N-\lambda)<\infty$.
As $K$ in \eqref{npertu} is a positive compact operator these facts 
remain true for $N_{\rm op}(\lambda)$ and $N(\lambda)$.
The assertions (ii) and (iii) follow easily from \eqref{mulker2}, \eqref{mulker3}, and a unique continuation argument.
Moreover, as $N_{\rm op}(\lambda)$ is compact and does not have finite rank,
we conclude that $\kappa_+(N(\lambda))=\infty$.
\end{proof}

\begin{remark}
We note that the domain of the relation $T$ in \eqref{trel} consists of all those 
$\varphi\in L^2(\cC)$ such that
$\gamma_N(\mu)\varphi\in\ran(A_N-\lambda)$.
Next, let $h_\lambda\in\ker(A_N-\lambda)$ and $\varphi\in L^2(\cC)$.
Then 
$$(\mu-\lambda)(\gamma_N(\mu)\varphi,h_\lambda)_{L^2(\Omega)}
=(\cL \gamma_N(\mu)\varphi,h_\lambda)_{L^2(\Omega)}-(\gamma_N(\mu)\varphi,A_N h_\lambda)_{L^2(\Omega)}
=-(\varphi,h_\lambda\vert_\cC)_{L^2(\cC)}$$
by Green's second identity and we used that $\partial_\nu h_\lambda\vert_\cC=0$.
Hence for all $\varphi\in L^2(\cC)$ we conclude that $\gamma_N(\mu)\varphi\in\ran(A_N-\lambda)$ if and only if
$\varphi\perp h_\lambda\vert_\cC$ for all $h_\lambda\in\ker(A_N-\lambda)$.
This is in accordance with the form of $\dom N(\lambda)$ 
in Theorem~\ref{thmdn}, i.e. 
$$
\dom T=\dom N(\lambda)
=\bigl\{\psi\in L^2(\cC):(\psi,f_\lambda\vert_\cC)_{L^2(\cC)}=0 \mbox{ for all }
   f_\lambda\in \ker (A_N-\lambda)\bigr\}.
$$
\end{remark}

In the next example we show that the estimate on the number of negative eigenvalues of $N(\lambda)$ in Theorem~\ref{dnmapsa}~(i) is not optimal.
Roughly speaking the reason is that eigenvalues of $A_D$ which are smaller than $\lambda$ lead to a cancellation of negative eigenvalues of $N(\lambda)$.

\begin{example}
Suppose that $\Omega=[0,1] \times [0,1]$ and that $\cL=-\Delta$ (that is $V=0$).
It is well-known and not difficult to see that the eigenvalues of
the Dirichlet Laplacian $A_D$ and the Neumann Laplacian $A_N$ are given by 
\begin{equation*}
\begin{split}
 \sigma_p(A_D)&=\bigl\{(m^2+n^2)\pi^2:m,n \in \{ 1,2,\dots \} \bigr\}\\
              &=\bigl\{2\pi^2,5\pi^2,5\pi^2,8\pi^2,10\pi^2,10\pi^2,13\pi^2,13\pi^2\dots \bigr\}
 \end{split}
 \end{equation*}
and
 \begin{equation*}
 \begin{split}
 \sigma_p(A_N)&=\bigl\{(m^2+n^2)\pi^2:m,n \in \{ 0,1,\dots \} \bigr\}\\
 &=\bigl\{0,\pi^2,\pi^2,2\pi^2,4\pi^2,4\pi^2,5\pi^2,5\pi^2,8\pi^2,\dots\bigr\}
 \end{split} 
\end{equation*}
respectively.
Hence for all $\lambda\in (4\pi^2,5\pi^2)$ the estimate in  Theorem~\ref{dnmapsa}~(i) becomes
\begin{equation}\label{estinot}
 \kappa_-(N(\lambda))\leq \kappa_-(A_N-\lambda)=6.
\end{equation}
However, it follows from Friedlander's inequality (see \cite{ArM2}, Proposition 4, and \cite{Friedlander}) that the Dirichlet-to-Neumann map $D(\lambda)$ has exactly 
\begin{equation*}
 \sharp\bigl\{\lambda_k\in\sigma_p(A_N):\lambda_k\leq \lambda \bigr\}-\sharp\bigl\{\mu_j\in\sigma_p(A_D):\mu_j\leq \lambda \bigr\}=6-1=5
\end{equation*}
eigenvalues in $(-\infty,0]$.
As $0$ is an eigenvalue of $D(\lambda)$ if and only if $\lambda$ is an eigenvalue of $A_N$ it follows that
in the present situation the Dirichlet-to-Neumann map $D(\lambda)$ has $5$ eigenvalues in $(-\infty,0)$.
Thus $N(\lambda)=D(\lambda)^{-1}$
also has $5$ eigenvalues in $(-\infty,0)$, i.e., the estimate \eqref{estinot} is not sharp.
\end{example}

The next theorem is a corollary of Theorem~\ref{dnmapsa}.
The Dirichlet-to-Neumann map $D(\lambda)$ as the inverse of the Neumann-to-Dirichlet map is selfadjoint in $L^2(\cC)$.
The nonzero eigenvalues of $D(\lambda)$ are the reciprokes of the nonzero eigenvalues of $N(\lambda)$, and $\ker D(\lambda)=\mul N(\lambda)$
and $\mul D(\lambda)=\ker N(\lambda)$ by \eqref{mulker2} and \eqref{mulker3}. In particular, the
operator part $D_{\rm op}(\lambda)$ is 
an unbounded operator with finitely many negative eigenvalues.

\begin{theorem}\label{dnmapsa2}
For all $\lambda\in\dR$ the Dirichlet-to-Neumann map $D(\lambda)$ is a selfadjoint relation in $L^2(\cC)$ defined on the subspace
\begin{equation*}
\dom D(\lambda)
= \bigl\{\varphi\in H^1(\cC):(\varphi,\partial_\nu f_\lambda\vert_\cC)=0 \mbox{ for all }
f_\lambda\in \ker (A_D-\lambda)\bigr\}\subset L^2(\cC)
\end{equation*}
with multivalued part $\mul D(\lambda)= \{\partial_\nu f_\lambda\vert_\cC:f_\lambda\in \ker (A_D-\lambda)\}$.
The operator part $D_{\rm op}(\lambda)$ of $D(\lambda)$
is an unbounded selfadjoint operator in the Hilbert space $\dom D(\lambda)$.
Moreover,
\begin{itemize}
 \item [{\rm (i)}] $\kappa_-(D(\lambda))=\kappa_-(N(\lambda))\leq \kappa_-(A_N-\lambda)<\infty$ and $\kappa_+(D(\lambda))=\infty$,
 \item [{\rm (ii)}] $\dim\ker(D(\lambda))=\dim\ker(A_N-\lambda)<\infty$,
 \end{itemize}
 and
 \begin{itemize}
 \item [{\rm (iii)}]$\dim\mul(D(\lambda))=\dim\ker(A_D-\lambda)<\infty$.
\end{itemize}
\end{theorem}

\appendix
\section{Linear relations} \label{Sappendix}

In this section we briefly recall some definitions and properties of linear relations in Hilbert spaces. 
A (closed) linear
relation $S$ from a Hilbert space $\cG$ into a Hilbert space $\cH$ is a (closed) 
subspace of $\cG\times\cH$.
The elements in a linear
relation $S$ consist of two components and will usually be written in the form $\{g,h\}\in S$.
The domain, range, kernel
and multivalued part of a linear relation $S$ from $\cG$ into $\cH$ are defined as
\begin{equation*}
 \begin{split}
  \dom S&=\bigl\{g\in\cG:\{g,h\}\in S \mbox{ for some } h\in\cH\bigr\},\\
  \ran S&=\bigl\{h\in\cH:\{g,h\}\in S \mbox{ for some } g\in\cG\bigr\},\\
  \ker S&=\bigl\{g\in\cG:\{g,0\}\in S\bigr\},
  \end{split}
  \end{equation*}
and
\begin{equation*}
\mul S =\bigl\{h\in\cH:\{0,h\}\in S\bigr\},
\end{equation*}
respectively.
Observe that a linear relation $S$ is the graph of an operator 
if and only if $\mul S=\{0\}$.
The inverse $S^{-1}$ of  a linear relation $S$ from $\cG$ to $\cH$ is defined by
\begin{equation*}
 S^{-1}=\bigl\{\{h,g\} \in \cH \times \cG:\{g,h\}\in S\bigr\},
\end{equation*}
and $S^{-1}$ is a linear relation from $\cH$ into $\cG$. Note that $S^{-1}$ is closed if and only if $S$ is closed.
Moreover,
it is easy to see that $\dom S=\ran S^{-1}$ and $\mul S=\ker S^{-1}$.
The sum $S+T$ of two linear relations $S$ and $T$ from
$\cG$ into $\cH$ is defined by
\begin{equation*}
S+T=\bigl\{\{g,h+h^\prime\}:\{g,h\}\in S \mbox{ and } \{g,h^\prime\}\in T\bigr\}. 
\end{equation*}
It is clear that $S+T$ is also a linear relation from $\cG$ to $\cH$.
Assume that $\cK$ is a further Hilbert space
and let $R$ be a linear relation from $\cK$ to $\cG$.
Then the product 
\begin{equation*}
SR=\bigl\{\{k,h\} \in \cK \times \cH :
   \mbox{there exists a } g \in \cG \mbox{ such that } \{k,g\}\in R \mbox{ and } \{g,h\}\in S\bigr\} 
\end{equation*}
is a linear relation from $\cK$ to $\cH$.
The adjoint $S^*$ of a linear relation $S$ from $\cG$ into $\cH$ is defined by
\begin{equation*}
S^*=\bigl\{\{h^\prime,g^\prime\} \in \cH \times \cG: (h,h^\prime)_{\cH}=(g,g^\prime)_{\cG}
   \mbox{ for all } \{g,h\}\in S\bigr\}.
\end{equation*}
This definition extends the usual definition of the adjoint of a bounded or unbounded operator.
Observe that $S^*$ is a closed linear relation from $\cH$ into $\cG$ and that 
$(S^*)^{-1}=(S^{-1})^*$ and $S^{**}=\overline S$, where $\overline S$ is the closure of $S$
in $\cG\times\cH$.
Moreover, 
it is not difficult to check that
\begin{equation}\label{dommuls}
(\ran S)^\bot=\ker S^*\qquad\text{and}\qquad (\dom S)^\bot=\mul S^*.
\end{equation}
From the second equality in \eqref{dommuls} it also follows that the adjoint of 
$S$ is an operator if and only if $\dom S$ is dense in $\cG$. In the case that $\cG\subset\cH\subset\cG^\prime$ form a rigging of Hilbert spaces
and $S$ is a linear relation from $\cG$ into $\cH$
the adjoint with respect to the extension of the inner product in $\cH$ onto $\cG\times\cG^\prime$ is denoted by $S^\prime$, which is a linear
relation from $\cH$ into $\cG^\prime$.

Assume now that $S$ is a closed linear relation in the Hilbert space $\cH$.
The point spectrum $\sigma_p(S)$ is defined as the set of all $\lambda\in\dC$
such that $\ker(S-\lambda)\not=\{0\}$.
An element $\lambda\in\dC$ belongs to the resolvent set $\rho(S)$ of $S$ if $(S-\lambda)^{-1}\in\cL(\cH)$.
The spectrum of $S$ is $\sigma(S) = \dC \setminus \rho(S)$.

A linear relation $A$ in $\cH$ is said to be symmetric, or essentially selfadjoint, or selfadjoint if
$A\subset A^*$, or $\overline A=A^*$, or $A=A^*$, respectively.
For a selfadjoint relation $A$ one has $(\dom A)^\bot=\mul A$ and it follows that
$A$ can be regarded as an orthogonal sum of a selfadjoint operator in the Hilbert space $\cH_{\rm op}=\overline{\dom A}$ and a purely
multivalued relation $A_\infty=\{\{0,h\}:h\in\mul A\}$ in the Hilbert space $\cH_\infty=\mul A$.
In particular, $\dC\setminus\dR\subset\rho(A)$
and $\sigma(A)\subset\dR$.
We will also make use of the fact that the sum $A+C$ of a selfadjoint relation $A$ in 
$\cH$ and a symmetric operator $C\in\cL(\cH)$
is a selfadjoint relation in $\cH$.

The following proposition provides a sufficient criterion for the selfadjointness of a 
certain product of a selfadjoint relation with two bounded operators.
This statement plays an important role in the proof of Theorem~\ref{dnmapsa}.

\begin{proposition}\label{ttprop}
Let $\cH$ and $\cG$ be Hilbert spaces,
let $A$ be a selfadjoint relation in $\cH$, let $B\in\cL(\cG,\cH)$, 
and assume that $\ran A$ is closed.
Then the relation
\begin{equation*}
T=B^*A^{-1}B 
\end{equation*}
is essentially selfadjoint in $\cG$.
If, in addition, $B^*\upharpoonright\ker A$ is boundedly invertible then
$T$ is selfadjoint in $\cG$.
\end{proposition}

\begin{proof}
Note first that the relation $T$ has the form
\begin{equation*}
 T=\bigl\{\{\varphi, B^* f\}: \varphi \in \cG, \; f \in \cH \mbox{ and } \{B\varphi,f\}\in A^{-1}\bigr\}
\end{equation*}
and that
\begin{equation*}
 \begin{split}
  \dom T&=\bigl\{\varphi\in\cG: B\varphi\in \dom A^{-1}=\ran A \bigr\},\\
  \mul T&=\bigl\{B^*f:f\in \mul A^{-1}=\ker A\bigr\}.
 \end{split}
\end{equation*}
Observe also that $\mul T$ is closed if $B^*\upharpoonright\ker A$ is boundedly invertible.
Furthermore, as $A$ is assumed to be selfadjoint the same is true for $A^{-1}$ and, in particular, $A^{-1}$ is symmetric.
This implies that $T$ and $\overline T$ are symmetric and hence
\begin{equation}\label{edommul}
\dom T\subset \dom \overline T\subset \dom T^*\qquad\text{and}\qquad \mul T\subset \mul \overline T\subset \mul T^*.
\end{equation}

We claim that 
\begin{equation}\label{einclttt}
 \bigl(\mul T\bigr)^\bot\subset\dom T\qquad\text{and}\qquad\bigl(\dom T\bigr)^\bot\subset \overline{\mul T}.
\end{equation}
In fact, for the first inclusion assume that $\psi\in\cG$ is orthogonal to $\mul T$.
Then we have 
\begin{equation*}
0=(\psi,B^*f)_{\cG}=(B\psi,f)_{\cH}
\end{equation*}
for all $f\in\ker A$ and hence $B\psi$ is orthogonal to $\ker A=(\ran A)^\bot$.
As $\ran A$ is assumed to be closed
we conclude that $B\psi\in\ran A$ and hence $\psi\in\dom T$.
This shows the first inclusion \eqref{einclttt}.
The second inclusion in \eqref{einclttt} follows by taking orthogonal complements.

We conclude from \eqref{dommuls}, \eqref{edommul}, and the second inclusion in \eqref{einclttt} that
\begin{equation*}
\mul \overline T \subset\mul T^* = \bigl(\dom T\bigr)^\bot\subset\overline{\mul T}\subset\mul\overline T,
\end{equation*}
and hence
\begin{equation}\label{gut1}
 \overline{\mul T}=\mul \overline T=\mul T^*.
\end{equation}
Similarly, from \eqref{dommuls}, \eqref{edommul}, the first inclusion in \eqref{einclttt}, and $\mul T\subset \mul \overline T$ we find
\begin{equation*}
\dom T \subset\dom\overline T\subset \dom T^*\subset\bigl(\dom T^*\bigr)^{\bot\bot}=\bigl(\mul \overline T\bigr)^\bot\subset(\mul T)^\bot\subset\dom T,
\end{equation*}
and hence 
\begin{equation}\label{gut2}
 \dom T=\dom\overline T= \dom T^*.
\end{equation}

The assertions now follow from \eqref{gut1} and \eqref{gut2}.
In fact, in order to show that $T$ is essentially selfadjoint it remains to
check that the inclusion $T^*\subset\overline T$ holds.
For this let $\{\varphi,\psi\}\in T^*$.
Then by \eqref{gut2} there exists a 
$\vartheta$ such that $\{\varphi,\vartheta\}\in \overline T$. 
As $\overline T$ is symmetric we have $\overline T\subset T^*$ and $\{\varphi,\vartheta\}\in T^*$. 
Then $\psi-\vartheta\in\mul T^*=\mul \overline T$ by \eqref{gut1} and we obtain $\{0,\psi-\vartheta\}\in \overline T$. 
Therefore 
$$\{ \varphi, \psi\}=\{ \varphi, \vartheta\}+\{0, \psi-\vartheta\}\in \overline T,$$
and hence $T$ is essentially selfadjoint.
If, in addition, $B^*\upharpoonright\ker A$ is boundedly invertible then $\mul T$ is 
closed and hence $\mul T=\mul T^*$ by \eqref{gut1}.
Now the argument above remains valid with $\overline T$ replaced by $T$ 
and it follows that $T$ is selfadjoint.
\end{proof}

\end{document}